\documentclass[review,onefignum,onetabnum]{siamart190516}

\usepackage{lipsum}
\usepackage{amsfonts}
\usepackage{graphicx}
\usepackage{epstopdf}
\usepackage{algorithmic}
\usepackage{amsmath}
\usepackage{amssymb}
\usepackage{bm}
\usepackage{booktabs}
\usepackage{multirow}
\ifpdf
  \DeclareGraphicsExtensions{.eps,.pdf,.png,.jpg}
\else
  \DeclareGraphicsExtensions{.eps}
\fi

\newsiamremark{remark}{Remark}
\newsiamremark{hypothesis}{Hypothesis}
\crefname{hypothesis}{Hypothesis}{Hypotheses}
\newsiamthm{claim}{Claim}
\newsiamthm{example}{Example}
\newsiamthm{fact}{Fact}

\headers{Convergence Rate Analysis for FP Iterations of GAN Operators}{Y. Lin and Y. Xu}

\title{Convergence Rate Analysis for Fixed-Point Iterations of Generalized Averaged Nonexpansive Operators}

\author{Yizun Lin\thanks{Department of Mathematics, College of Information Science and Technology, Jinan University, Guangzhou 510632, China
  (\email{linyizun@jnu.edu.cn}). Supported in part by Fundamental Research Funds for the Central Universities of China under Grant 11620352, by the Opening Project of Guangdong Province Key Laboratory of Computational Science at the Sun Yat-sen University under Grant 2021006, and by Natural Science Foundation of China under Grant 11771464.}
\and Yuesheng Xu\thanks{Department of Mathematics and Statistics, Old Dominion University, Norfolk, VA 23529, USA
  (\email{y1xu@odu.edu}). This author is a Professor Emeritus of Mathematics at Syracuse University, Syracuse, New York. Supported in part by US National Science Foundation under grant DMS-1912958 and by Natural Science Foundation of China under grant 11771464. All correspondence should be addressed to this author.}
}

\usepackage{amsopn}

\DeclareMathOperator*{\argmin}{argmin}

\def \bR {\mathbb R}
\def \bN {\mathbb N}

\def \mI {\mathcal{I}}
\def \mN {\mathcal{N}}

\def \bRm {\bR^{m}}
\def \bRn {\bR^{n}}

\def \bRnn {\bR^{n\times n}}

\def \bRmn {\bR^{m\times n}}

\def \prox {\mathrm{prox}}

\def \Fix {\text{Fix}}
\def \FixT {\text{Fix}(T)}

\def \hx {\hat{x}}

\def \ymuSNset {\Omega_{\mu}^{\gamma}}

\ifpdf
\hypersetup{
  pdftitle={An Example Article},
  pdfauthor={D. Doe, P. T. Frank, and J. E. Smith}
}
\fi

\begin{document}

\maketitle

\begin{abstract}
We estimate convergence rates for fixed-point iterations of a class of nonlinear operators which are partially motivated from solving convex optimization problems. We introduce the notion of the generalized averaged nonexpansive (GAN) operator with a positive exponent, and provide a convergence rate analysis of the fixed-point iteration of the GAN operator. The proposed generalized averaged nonexpansiveness is weaker than the averaged nonexpansiveness while stronger than nonexpansiveness. We show that the fixed-point iteration of a GAN operator with a positive exponent converges to its fixed-point and estimate the local convergence rate (the convergence rate in terms of the distance between consecutive iterates) according to the range of the exponent. We prove that the fixed-point iteration of a GAN operator with a positive exponent strictly smaller than 1 can achieve an exponential global convergence rate (the convergence rate in terms of the distance between an iterate and the solution). Furthermore, we establish the global convergence rate of the fixed-point iteration of a GAN operator, depending on both the exponent of generalized averaged nonexpansiveness and the exponent of the H$\ddot{\text{o}}$lder regularity, if the GAN operator is also H$\ddot{\text{o}}$lder regular. We then apply the established theory to three types of convex optimization problems that appear often in data science to design fixed-point iterative algorithms for solving these optimization problems and to analyze their convergence properties.
\end{abstract}



\section{Introduction}
We consider in this paper the convergence rate analysis of fixed-point algorithms. Fixed-point type algorithms have been popular in solving nondifferentiable convex or nonconvex optimization problems such as image processing \cite{chen2013primal,li2015multi,lu2016multiplicative,micchelli2011proximity,micchelli2013proximity,shen2016wavelet}, medical imaging \cite{krol2012preconditioned,lin2019krasnoselskii,ross2017relaxed,zheng2019sparsity}, machine learning \cite{chen2015convergence,li2018fixed,li2019two,polson2015proximal}, and compressed sensing \cite{figueiredo2007gradient,zhu2015fast}. Existing fixed-point type algorithms for optimization including the gradient descent algorithm \cite{bertsekas2015convex,ruder2016overview}, the proximal point algorithm \cite{rockafellar1976monotone}, the proximal gradient algorithm \cite{bertsekas1976goldstein,parikh2014proximal}, the forward-backward splitting algorithm \cite{chen1997convergence,tseng2000modified} and the fixed-point proximity algorithm \cite{li2015multi,lin2019krasnoselskii,micchelli2011proximity,micchelli2013proximity}.

Traditionally, fixed-point algorithms were often developed by constructing contractive operators (contraction mapping) or averaged nonexpansive operators  \cite{agarwal2001fixed, bauschke2017convex,micchelli2011proximity}. Such constructions bring advantages for fixed-point algorithms. It makes the convergence analysis more straightforward and provides robust and monotonic convergence. That is, as the fixed-point iteration proceeds, the distance between the iterate  and the true solution is monotonically decreasing. In addition, fixed-point algorithms are comparatively simple and easy to implement. Most optimization problems in real-world applications may be reformulated as fixed-point equations of averaged nonexpansive operators but  usually not contractive operators. It is also known \cite{baillion1996rate} that the local convergence rate (the convergence rate in terms of the distance between consecutive iterates) of the fixed-point iteration of an averaged nonexpansive operator is $o(k^{-\frac{1}{2}})$, where $k$ denotes the step of the iteration. However, for certain problems, the operators that result in the fixed-point reformulation are not averaged nonexpansive. For such fixed-point iterations, the existing theory of the averaged nonexpansive operator is not applicable. Therefore, there is a need to extend the existing results.
We are interested in understanding the following two issues:  Is there a class of operators, satisfying a condition weaker than the averaged nonexpansiveness, whose  fixed-point iterations still converge? Is there a subclass of the averaged nonexpansive operators whose fixed-point iterations have convergence rates higher than order $o(k^{-\frac{1}{2}})$?
For the first issue, some classes of operators were proposed, such as demicontracitve operators \cite{hicks1977mann,muarucster2011strong} and quasi-firmly type nonexpansive operators \cite{song2009halpern,song2011successive}. However, these classes of operators do not ensure the closeness of the composition operation, which makes them not applicable to a large range of real-world optimization problems. In addition, their fixed-point iterations do not have a  convergence rate higher than that the averaged nonexpansive operators have.

To address these two issues, we introduce the notion of the generalized averaged nonexpansive (GAN) operator with a positive exponent $\gamma$, establish the convergence property of the fixed-point iterations of a GAN operator and prove their convergence rate higher than the known result for a range of the exponent $\gamma$. Specifically, this notion generalizes the averaged nonexpansive operators in two aspects. First, the generalized averaged nonexpansiveness with exponent $\gamma$ of an operator for $\gamma>2$ is weaker than the averaged nonexpansiveness which corresponds to $\gamma=2$, but it still guarantees convergence of its fixed-point iterations. Second, the exponent $\gamma$ allows us to refine the local convergence rates of the resulting fixed-point iterations, leading to a local convergence rate higher than that the averaged nonexpansive operator has.

We organize this paper in seven sections. In section 2, we describe fixed-point formulations for three convex optimization models. We introduce in section 3 the notion of GAN operator and study its connection with nonexpansive, averaged nonexpansive and contractive operators. Several basic properties of GAN operators are also provided.  Sections 4 and 5 are respectively devoted to local and global convergence rate analysis of fixed-point iterations of  GAN operators. In section 6, we employ the convergence rate results developed in Sections 4 and 5 to analyze the convergence rate of the fixed-point algorithms for three convex optimization models described in Section 2. Section 7 offers a conclusion.

\section{Fixed-point formulations for optimization}\label{sectFPopt}

Solutions of optimization problems are often formulated as fixed-points of nonlinear operators. Such formulations have great advantages for algorithm development and convergence analysis. 
We describe in this section fixed-point formulations for convex optimization problems.

By $\Gamma_0(\bRn)$ we denote the class of all proper lower semicontinuous convex functions from $\bRn$ to $\bR\cup\{\infty\}$. We assume that  $\Psi\in\Gamma_0(\bRn)$ and consider the convex optimization problem 
\begin{equation}\label{minmodel}
\argmin_{x\in\bRn}\Psi(x).
\end{equation}
Throughout this paper, we assume that the objective function $\Psi\in\Gamma_0(\bRn)$ has at least one minimizer without further mentioning. 
Solutions of problem \eqref{minmodel} may be reformulated as fixed-points of certain operators, depending on the smoothness of the objective function $\Psi$. To this end, we first recall the notions of the proximity operator and subdifferential of a convex function. Let $H\in\bR^{n\times n}$ be a symmetric positive definite matrix. For $x\in\bRn$ and $y\in\bRn$, we define the $H$-weighted inner product by $\langle x,y\rangle_H:=x^\top Hy$ and the corresponding $H$-weighted norm by $\|x\|_H:=\langle x,x\rangle_H^\frac{1}{2}$. Then the $\ell_2$ inner product and $\ell_2$ norm are given by $\langle x,y\rangle_2:=\langle x,y\rangle_I$ and $\|x\|_2:=\|x\|_I$ respectively, where $I\in\bR^{n\times n}$ denotes the identity matrix. Let $\psi\in\Gamma_0(\bRn)$. The proximity operator of $\psi$ at $x\in\bRn$ is defined by
$$
\prox_\psi(x):=\argmin_{u\in\bRn}\left\{\frac{1}{2}\|u-x\|_2^2+\psi(u)\right\}.
$$
The subdifferential of $\psi$ at $x\in\bRn$ is defined by
$$
\partial \psi(x):=\{y\in\bRn:\psi(z)\geq\psi(x)+\langle y,z-x\rangle_2\ \mbox{for all}\ z\in\bRn\}.
$$

We list below examples of the operators derived from problem \eqref{minmodel} for different types of objective functions. In the following three cases, we assume that function $f\in\Gamma_0(\bRn)$ is differentiable with an $L$-Lipschitz continuous gradient with respect to $\|\cdot\|_2$. We let $\bR_+$ denote the set of all positive real numbers throughout the paper.

\vspace{0.5em}
Case 1. $\Psi:=f$. In this case, a minimizer of \eqref{minmodel} is identified as a fixed-point of operator
\begin{equation}\label{defT1}
T_1:=\mI-\beta\nabla f, \ \ \mbox{where}\ \ \beta\in\bR_+. 
\end{equation}
We will call $T_1$ a gradient descent operator. This type of optimization problems has important applications in machine learning (e.g. smoothed SVM, ridge regression) \cite{zhang2004solving} and medical imaging \cite{ahn2003globally,fessler1994penalized}.

\vspace{0.5em}
Case 2. $\Psi:=f+g$, where $g\in\Gamma_0(\bRn)$ may not be differentiable, but has a closed form of its proximity operator. By using Fermat's rule (Theorem 16.3 of \cite{bauschke2017convex}) and a relation between the subdifferential and the proximity operator (Proposition 2.6 of \cite{ micchelli2011proximity}), a minimizer of \eqref{minmodel} is identified as a fixed-point of operator
\begin{equation}\label{defT2}
T_2:=\prox_{\beta g}\circ(\mI-\beta\nabla f), \ \ \mbox{where}\ \ \beta\in\bR_+.
\end{equation}
Obviously, $T_2=\prox_{\beta g}\circ T_1$. This type of optimization models is raised from machine learning (e.g. $\ell_1$-SVM, LASSO regression) \cite{li2019two}, compressed sensing \cite{figueiredo2007gradient} and image processing \cite{beck2009fast,figueiredo2003algorithm}.

\vspace{0.5em}
Case 3: $\Psi=f+g\circ B+h$, where $g\in\Gamma_0(\bRm)$ and $h\in\Gamma_0(\bRn)$ have closed forms of their proximity operators and $B\in\bRmn$ is a matrix. Let $g^*$ denote the conjugate function of $g$, that is, 
$$
g^*(z):=\sup_{y\in\bRm}\{\langle z,y\rangle_2-g(y)\},\ \ \mbox{for}\ \ z\in\bRm.
$$
By using Fermat's rule, the chain rule of subdifferential, a relation between the subdifferential and the proximity operator, and introducing a dual variable, a minimizer of \eqref{minmodel} in this case can be identified as a fixed-point of a nonlinear operator. Specifically, we let
 $v:=\left(\begin{array}{c}
x\\
y
\end{array}\right)$, for $x\in\bRn$, $y\in\bRm$, and introduce $r:\bR^{n+m}\to\bR$ by $r(v):=f(x)$, $\widetilde{T}:\bR^{n+m}\to\bR^{n+m}$
by $\widetilde{T}(v):=\left(\begin{array}{c}
\prox_{\beta h}(x)\\
\prox_{\eta g^*}(y)
\end{array}\right)$, where $\beta$ and $\eta$ are two positive parameters. Let 
$$
E:=\left(\begin{array}{cc}
I_n&-\beta B^\top\\
\eta B&I_m
\end{array}\right), \ \ G:=\left(\begin{array}{cc}
I_n&-\beta B^\top\\
-\eta B&I_m
\end{array}\right),\ \
W:=\left(\begin{array}{cc}
\frac{1}{\beta}I_n&-B^\top\\
-B&\frac{1}{\eta}I_m
\end{array}\right).
$$
We then define the operators
$$
T_G:u\to\left\{v:(u,v)\ \text{satisfies that}\ v=\widetilde{T}\left((E-G)v+Gu\right)\right\}
$$
and
\begin{equation}\label{defT3}
T_3:=T_{G}\circ(\mI-W^{-1}\nabla r),\ \ \mbox{where}\ \ \beta,\eta\in\bR_+.
\end{equation}
It can be verified that if $v\in \bR^{n+m}$ is a fixed-point of $T_3$, then the corresponding $x\in \bRn$ is a minimizer of  \eqref{minmodel}. One can refer to \cite{li2016fast,lin2019krasnoselskii} for more details of the derivation of operator $T_3$.
The model in this case has applications in image processing \cite{cai2008framelet,chan2003wavelet,rudin1992nonlinear}, machine learning \cite{sra2012optimization} and medical imaging \cite{komodakis2015playing,krol2012preconditioned, li2016fast,lin2019krasnoselskii}.

Analysis for convergence and convergence rate of fixed-point algorithms can be done by analyzing properties of the operators that define the fixed-point iterations. It is known \cite{baillion1996rate,browder1967construction} that a fixed-point iteration of an averaged nonexpansive operator converges to its fixed-point with a local convergence rate $o\left(k^{-\frac{1}{2}}\right)$. There are operators from application which may not be averaged nonexpansive.
Aiming at relaxing the averaged nonexpansiveness condition for analysis of convergence and convergence rates of fixed-point iterations of operators, we introduce the notion of the generalized averaged nonexpansive operator and show that the fixed-point iterations of such an operator are convergent and have certain convergence rates.

\section{Generalized averaged nonexpansive operators}\label{sec_SNoperator}

In this section, we introduce the notion of the generalized averaged nonexpansive (GAN) operator and study its connection with the  nonexpansive, averaged nonexpansive and contractive operators. Several basic properties of GAN operators are also provided.

We first describe the definition of GAN operator. Let $\mI$ denote the identity operator. 

\begin{definition}\label{def_SN} Let $\|\cdot\|$ denote a norm on $\bRn$.
An operator $T:\bRn\to\bRn$ is said to be generalized averaged nonexpansive if there exist $\gamma, \mu\in\bR_+$ such that
\begin{equation}\label{neq_SNdef}
\|Tx-Ty\|^\gamma+\mu\|(\mI-T)x-(\mI-T)y\|^\gamma\leq\|x-y\|^\gamma,\ \ \mbox{for all}\ \ x,y\in\bRn.
\end{equation}
Specifically, we say that $T$ is $\mu$-generalized averaged nonexpansive ($\mu$-GAN) with exponent $\gamma$ with respect to $\|\cdot\|$.
\end{definition}

The norm $\|\cdot\|$ mentioned in Definition \ref{neq_SNdef} can be any norm including the norm induced by an inner product, weighted inner product and the $\ell_1$ norm.
According to Definition \ref{def_SN}, for $\mu_1>\mu_2>0$, if $T$ is $\mu_1$-GAN with exponent $\gamma\in\bR_+$, then it is also $\mu_2$-GAN with exponent $\gamma$.

Let $\FixT$ denote the set of all fixed-points of operator $T$ and
$$
\Lambda:=\{T:\bRn\to\bRn|\ \FixT\neq\varnothing\}.
$$
Throughout this paper, we will assume that $T\in\Lambda$ without further mentioning. It follows from Definition \ref{def_SN} that if $T$ is GAN, then
\begin{equation}\label{neq2_SNdef}
\|Tx-\hx\|^\gamma+\mu\|Tx-x\|^\gamma\leq\|x-\hx\|^\gamma,\ \ \mbox{for all}\ \ x\in\bRn,\ \hx\in\FixT.
\end{equation}

We next discuss connections of the GAN operators with the nonexpansive, averaged nonexpansive, firmly nonexpansive and contractive operators. For notational simplicity, throughout the remaining part of this paper, we use $\langle\cdot, \cdot\rangle$ and $\|\cdot\|$ to represent a weighted inner product and the corresponding weighted norm with respect to a symmetric positive definite matrix, respectively, unless there is a need to specify the weight matrix. An operator $T:\bRn\to\bRn$ is called nonexpansive if
$$
\|Tx-Ty\|\leq\|x-y\|,\ \ \mbox{for all}\ \  x,y\in\bRn,
$$
and is called firmly nonexpansive if
$$
\|Tx-Ty\|^2\leq\langle Tx-Ty,x-y\rangle,\ \ \mbox{for all}\ \ x,y\in\bRn.
$$
If there exists a nonexpansive operator $\mN:\bRn\to\bRn$ and $\alpha\in(0,1)$ such that $T=(1-\alpha)\mI+\alpha\mN$, we say that $T$ is $\alpha$-averaged nonexpansive. If there exists $\rho\in(0,1)$ such that
\begin{equation}\label{contractineq}
\|Tx-Ty\|\leq\rho\|x-y\|, \ \ \mbox{for all}\ \ x,y\in\bRn,
\end{equation}
we say that $T$ is contractive ($\rho$-contractive). From Definition \ref{def_SN}, we can immediately see that GAN operators are nonexpansive.

To see the connection of the generalized averaged nonexpansiveness with the averaged nonexpansiveness, we recall a known result (Proposition 4.35 of \cite{bauschke2017convex}).

\begin{proposition}\label{prop_averaged}
Let $\alpha\in(0,1)$. Operator $T:\bRn\to\bRn$ is $\alpha$-averaged nonexpansive if and only if
\begin{equation}\label{neq_ANandSN}
\|Tx-Ty\|^2+\frac{1-\alpha}{\alpha}\|(\mI-T)x-(\mI-T)y\|^2\leq\|x-y\|^2, \ \ \mbox{for all}\ \ x,y\in\bRn.
\end{equation}
\end{proposition}

Proposition \ref{prop_averaged} implies that the $\alpha$-averaged nonexpansiveness is equivalent to the $\frac{1-\alpha}{\alpha}$-generalized averaged nonexpansiveness with exponent 2. In particular, the firm nonexpansiveness is equivalent to the 1-generalized averaged nonexpansiveness with exponent 2, since it is also equivalent to the $\frac{1}{2}$-averaged nonexpansiveness (see Remark 4.34 of \cite{bauschke2017convex}). We will show later in this section that for any given $\gamma\in\bR_+$, a contractive operator must be GAN with exponent $\gamma$. The generalization from averaged nonexpansiveness to generalized averaged nonexpansiveness will lead to higher order convergence rate for fixed-point algorithm defined by a GAN operator with an exponent smaller than 2. We will discuss this point in a later section.

We now study the relation among the GAN operators with different exponents and the relation among the generalized averaged nonexpansiveness, contractivity and FP-contractivity (which we will define later). To this end, we first establish a technical lemma.

\begin{lemma}\label{gammapowerineq}
Let $a$, $b$ and $c$ be three nonnegative real numbers, $\gamma\in\bR_+$. Then the following statements hold:
\begin{itemize}
\item[$(i)$] If $\gamma>1$, then $(a+b)^\gamma\geq a^\gamma+b^\gamma$.
\item[$(ii)$] If $\gamma'>\gamma$ and $a^{\gamma}+b^{\gamma}\leq c^{\gamma}$, then $a^{\gamma'}+b^{\gamma'}\leq c^{\gamma'}$.
\end{itemize}
\end{lemma}
\begin{proof}
We first prove $(i)$. To this end, we define $\psi(t):=(1+t)^\gamma-(1+t^\gamma)$ and $\phi(t):=t^{\gamma-1}$, $t\in[0,+\infty)$. Then $\psi'(t)=\gamma\left((1+t)^{\gamma-1}-t^{\gamma-1}\right)$. If $\gamma>1$, since $\phi$ is strictly increasing on $[0,+\infty)$, we know that $\psi'(t)>0$, and hence $\psi$ is strictly increasing on $[0,+\infty)$. Thus $\psi(t)\geq\psi(0)=0$ for $t\in[0,+\infty)$. It is obvious that $(a+b)^\gamma\geq a^\gamma+b^\gamma$ holds for $b=0$. For the case $b>0$, we have $\psi\left(\frac{a}{b}\right)=\left(1+\frac{a}{b}\right)^\gamma-\left(1+\left(\frac{a}{b}\right)^\gamma\right)\geq0$, which implies that $(a+b)^\gamma-(a^\gamma+b^\gamma)\geq0$.

Now we employ $(i)$ to prove $(ii)$. Since $a^{\gamma}+b^{\gamma}\leq c^{\gamma}$ and $\frac{\gamma'}{\gamma}>1$, by writing $a^{\gamma'}+b^{\gamma'}={(a^{\gamma})}^{\frac{\gamma'}{\gamma}}+{(b^{\gamma})}^{\frac{\gamma'}{\gamma}}$ and using $(i)$, we have that
$$
a^{\gamma'}+b^{\gamma'}\leq (a^{\gamma}+b^{\gamma})^{\frac{\gamma'}{\gamma}}\leq (c^{\gamma})^{\frac{\gamma'}{\gamma}}=c^{\gamma'},
$$
which completes the proof.
\end{proof}

We establish the inclusion relation of the class of GAN operators with different exponents in the following proposition.

\begin{proposition}\label{prop_SNinclrel}
If $0<\gamma_1<\gamma_2$ and operator $T:\bRn\to\bRn$ is GAN with exponent $\gamma_1$, then $T$ is GAN with exponent $\gamma_2$.
\end{proposition}
\begin{proof}
Since $T$ is GAN with exponent $\gamma_1$, there exists $\mu\in\bR_+$ such that
$$
\|Tx-Ty\|^{\gamma_1}+\mu\|(\mI-T)x-(\mI-T)y\|^{\gamma_1}\leq\|x-y\|^{\gamma_1}, \ \  \mbox{for all} \ \ x,y\in\bRn.
$$
Applying Lemma \ref{gammapowerineq} $(ii)$ with $a:=\|Tx-Ty\|$, $b:=\mu^{\frac{1}{\gamma_1}}\|(\mI-T)x-(\mI-T)y\|$, $c:=\|x-y\|$, $\gamma:=\gamma_1$ and $\gamma':=\gamma_2$, we obtain that
$$
\|Tx-Ty\|^{\gamma_2}+\mu^{\frac{\gamma_2}{\gamma_1}}\|(\mI-T)x-(\mI-T)y\|^{\gamma_2}\leq\|x-y\|^{\gamma_2}, $$
which implies that $T$ is GAN with exponent $\gamma_2$.
\end{proof}

By the above proof, we can also know that if $\mu\geq1$, then $\mu$-generalized averaged nonexpansiveness with exponent $\gamma_1$ implies $\mu$-generalized averaged nonexpansiveness with exponent $\gamma_2$. 

We next show that contractivity implies generalized averaged nonexpansiveness.

\begin{proposition}\label{prop_contrsubSN}
If operator $T:\bRn\to\bRn$ is $\rho$-contractive for some $\rho\in(0,1)$, then it is $\hat{\rho}$-GAN with exponent $\gamma$, where $\gamma\in\bR_+$ is an arbitrarily fixed number and $\hat{\rho}:=\frac{1-\rho^\gamma}{(1+\rho)^\gamma}$.
\end{proposition}
\begin{proof}
Since $T$ is $\rho$-contractive, for any fixed $\gamma\in \bR_+$, we have that
\begin{equation}\label{CONTRACT}
\|Tx-Ty\|^\gamma\leq\rho^\gamma\|x-y\|^\gamma,\ \ \text{for all}\ \ x,y\in\bRn.
\end{equation}
We choose $\hat{\rho}:=\frac{1-\rho^\gamma}{(1+\rho)^\gamma}$ and observe that $\rho^\gamma=1-\hat{\rho}(1+\rho)^\gamma$. We thus obtain that
$$
\rho^\gamma\|x-y\|^\gamma
=\|x-y\|^\gamma-\hat{\rho}(\|x-y\|+\rho\|x-y\|)^\gamma.
$$
By the definition \eqref{contractineq} of the contractive operator and the triangle inequality, we find that
$$
\|x-y\|+\rho\|x-y\|\geq\|x-y\|+\|Tx-Ty\|\geq\|(\mI-T)x-(\mI-T)y\|.
$$
Substituting this inequality into the right hand side of the above equation and then combining with \eqref{CONTRACT}, we get that
$$
\|Tx-Ty\|^\gamma\leq\|x-y\|^\gamma-\hat{\rho}\|(\mI-T)x-(\mI-T)y\|^\gamma,
$$
which proves the desired result.
\end{proof}

Proposition \ref{prop_contrsubSN} provides the inclusion of the class of contractive operators in the class of GAN operators with exponent $\gamma$ for any $\gamma\in\bR_+$. Moreover, the class of contractive operators is a proper subset of the class of GAN operators (see, Example \ref{SNproxabs} to be presented later).

We next investigate the inclusion relation of the class of FP-contractive operators and the class of GAN operators with exponent $\gamma\in(0,1)$. We now  define the FP-contractive operator. For $T\in\Lambda$, if there exists $\rho\in(0,1)$ such that
\begin{equation}\label{neq_QCdef}
\|Tx-\hx\|\leq\rho\|x-\hat{x}\|, \ \  \mbox{for all}\ \ x\in\bRn\backslash\FixT, \ \hx\in\FixT,
\end{equation}
then we say that $T$ is $\rho$-contractive with respect to its fixed-point set (or FP-$\rho$-contractive). From the definition of the FP-contractivity, contractive operators must be FP-contractive. However, an FP-contractive operator may not be contractive. For example, the identity operator $\mI$ is FP-contractive but not contractive. In addition, the fixed-point of a FP-contractive operator may not be unique. 

We need a technical lemma on the monotonicity of the function $\psi$ defined below. For $\gamma\in\bR_+$, let 
\begin{equation}\label{psialpha}
\psi(\alpha):=\frac{1-\alpha^\gamma}{(1-\alpha)^\gamma}, \ \ \alpha\in[0,1). 
\end{equation}

\begin{lemma}\label{lemma_SNsubQC}
If $\gamma<1$, then $\psi$ is strictly decreasing on $(0,1)$ and $\lim_{\alpha\to1^-}\psi(\alpha)=0$. If $\gamma>1$, then $\psi$ is strictly increasing on $(0,1)$ and $\lim_{\alpha\to1^-}\psi(\alpha)=\infty$.
\end{lemma}
\begin{proof}
It follows from the definition of $\psi$ that
$$
\psi'(\alpha)=\frac{\gamma}{(1-\alpha)^{\gamma+1}}(1-\alpha^{\gamma-1}),\ \ \alpha\in(0,1),
$$
which is negative for $\gamma<1$ and positive for $\gamma>1$. Hence, $\psi$ is strictly decreasing (resp., increasing) on $(0,1)$ if $\gamma<1$ (resp., $\gamma>1$). We now consider $\lim_{\alpha\to1^-}\psi(\alpha)$. By L'Hospital's Rule, we have that
$$
\lim_{\alpha\to1^-}\psi(\alpha)=\lim_{\alpha\to1^-}\left(\frac{\alpha}{1-\alpha}\right)^{\gamma-1},
$$
which is equal to $0$ for $\gamma<1$ and equal to infinity for $\gamma>1$.
\end{proof}

\begin{proposition}\label{prop_SNsubQC}
If $T\in\Lambda$ is GAN with exponent $\gamma$ for some $\gamma\in(0,1)$, then it is FP-$\rho$-contractive for some $\rho\in(0,1)$.
\end{proposition}
\begin{proof}
We prove this proposition by contradiction. Assume to the contrary that $T$ is not FP-$\rho$-contractive for any $\rho\in(0,1)$. That is, for any $\rho\in(0,1)$, there exist $x\in\bRn\backslash\FixT$ and $\hx\in\FixT$ such that $\|Tx-\hx\|>\rho\|x-\hx\|$. We next prove that $T$ is not GAN with exponent $\gamma$ for any $\gamma\in(0,1)$, that is, for any $\gamma\in(0,1)$ and any $\mu\in\bR_+$, there exist $x\in\bRn\backslash\FixT$ and $\hx\in\FixT$ such that
\begin{equation}\label{SNFPcontrConIn}
\|Tx-\hx\|^\gamma+\mu\|Tx-x\|^\gamma>\|x-\hx\|^\gamma.
\end{equation}

By Lemma \ref{lemma_SNsubQC}, for any $\gamma\in(0,1)$, $\psi$ defined by \eqref{psialpha} is continuous and strictly decreasing on $(0,1)$, and $\lim_{\alpha\to1^-}\psi(\alpha)=0$. This ensures that for any $\mu>0$, there exists $\rho_{\mu,\gamma}\in(0,1)$ such that $\mu>\psi(\rho_{\mu,\gamma})$. By the contradiction hypothesis, for this $\rho_{\mu,\gamma}$, there exist some $x\in\bRn\backslash\FixT$ and $\hx\in\FixT$ such that $\|Tx-\hx\|>\rho_{\mu,\gamma}\|x-\hx\|$. We next prove that \eqref{SNFPcontrConIn} holds.

Since $x\in\bRn\backslash\FixT$, we know that $\|Tx-x\|>0$. If $\|Tx-\hx\|\geq\|x-\hx\|$,  it is clear that \eqref{SNFPcontrConIn} holds since $\mu\|Tx-x\|^\gamma>0$. If $\|Tx-\hx\|<\|x-\hx\|$, we choose $\rho_{\mu,\gamma}':=\|Tx-\hx\|/\|x-\hx\|$. We then observe that $\rho_{\mu,\gamma}'\in(\rho_{\mu,\gamma},1)$ and satisfies
\begin{equation}\label{Txhxeqrhomu}
\|Tx-\hx\|=\rho_{\mu,\gamma}'\|x-\hx\|,
\end{equation}
and $\mu>\psi(\rho_{\mu,\gamma}')$, due to the monotone decreasingness of $\psi$ on $(0,1)$. This inequality together with $\|Tx-x\|>0$ and the definition of $\psi$ implies that 
$$
\mu\|Tx-x\|^\gamma>\psi(\rho_{\mu,\gamma}')\|Tx-x\|^\gamma=\frac{1-{\rho_{\mu,\gamma}'}^\gamma}{(1-\rho_{\mu,\gamma}')^\gamma}\|Tx-x\|^\gamma.
$$
This inequality combined with the triangle inequality $\|x-\hx\|\leq \|Tx-x\|+\|Tx-\hx\|$ and \eqref{Txhxeqrhomu} yields that
$$
\mu\|Tx-x\|^\gamma>\frac{1-{\rho_{\mu,\gamma}'}^\gamma}{(1-\rho_{\mu,\gamma}')^\gamma}(\|x-\hx\|-\|Tx-\hx\|)^\gamma=\left(1-{\rho_{\mu,\gamma}'}^\gamma\right)\|x-\hx\|^\gamma.
$$
Combining the above inequality and  \eqref{Txhxeqrhomu} leads to
$$
\|Tx-\hx\|^\gamma+\mu\|Tx-x\|^\gamma>(\rho_{\mu,\gamma}'\|x-\hx\|)^\gamma+\left(1-{\rho_{\mu,\gamma}'}^\gamma\right)\|x-\hx\|^\gamma=\|x-\hx\|^\gamma.
$$
This is \eqref{SNFPcontrConIn}, a contradiction to the generalized averaged nonexpansiveness of $T$ with exponent $\gamma$ for some $\gamma\in(0,1)$. Therefore, we complete the proof that $T$ is FP-$\rho$-contractive for some $\rho\in(0,1)$.
\end{proof}

We next demonstrate by a one dimensional example that the class of contractive operators is a proper subset of the class of GAN operators with exponent 1. To this end, we first establish a technical lemma. We mention here that a one-dimensional operator $T:\bR\to\bR$ is said to be monotonically increasing if $Tx\geq Ty$ for any $x,y\in\bR$ satisfying that $x>y$.

\begin{lemma}\label{examplemma}
If operator $T:\bR\to\bR$ is nonexpansive and monotonically increasing, then it is GAN with exponent 1.
\end{lemma}
\begin{proof}
It suffices to prove that for all $t_1,t_2\in\bR$,
\begin{equation}\label{ineq1_examplemma}
|T(t_1)-T(t_2)|+|(t_1-t_2)-(T(t_1)-T(t_2))|\leq|t_1-t_2|.
\end{equation}
If $t_1=t_2$, \eqref{ineq1_examplemma} clearly holds. Without loss of generality, we prove that \eqref{ineq1_examplemma} holds for the case $t_1>t_2$. In this case, we know that $T(t_1)\geq T(t_2)$ since $T$ is monotonically increasing. Furthermore, the nonexpansiveness of $T$ implies that $T(t_1)-T(t_2)\leq t_1-t_2$. Therefore,
$$
|T(t_1)-T(t_2)|+|(t_1-t_2)-(T(t_1)-T(t_2))|=|t_1-t_2|,
$$
which completes the proof.
\end{proof}

\begin{example}\label{SNproxabs}
Let $\lambda\in\bR_+$ and $T:=\prox_{\lambda|\cdot|}$. Then $T$ is GAN with exponent 1, but it is not GAN with exponent $\gamma$ for any $\gamma\in(0,1)$ and nor contractive.
\end{example}
\begin{proof}
Note that $T$ is firmly nonexpansive \cite{combettes2005signal}, and hence it is nonexpansive. It follows from Example 2.3 in \cite{micchelli2011proximity} that
$$
T(x)=\begin{cases}
x-\lambda,&\text{if}\ x>\lambda,\\
0,&\text{if}\ -\lambda\leq x\leq\lambda,\\
x+\lambda,&\text{if}\ x<-\lambda,
\end{cases}
$$
which is monotonically increasing. Then we conclude from Lemma \ref{examplemma} that $T$ is GAN with exponent $1$.

We next show that $T$ is not GAN with exponent $\gamma$ for all $\gamma\in(0,1)$. Suppose that there exists some $\gamma\in(0,1)$ such that $T$ is GAN with exponent $\gamma$. Since $\FixT=\{0\}$, there exists $\mu\in\bR_+$ such that
$$
|Tx-0|^\gamma+\mu|Tx-x|^\gamma\leq|x-0|^\gamma,\ \ \mbox{for all}\ \ x\in\bR,
$$
that is,
\begin{equation}\label{ineq_absproxexam}
\mu|Tx-x|^\gamma\leq|x|^\gamma-|Tx|^\gamma\ \ \mbox{for all}\ \ x\in\bR.
\end{equation}
Since $Tx=x-\lambda$ for $x>\lambda$, we have $|Tx-x|^\gamma=\lambda^\gamma$. Then \eqref{ineq_absproxexam} implies that
\begin{equation}\label{ineq2_absproxexam}
\mu\leq\lambda^{-\gamma}\left[x^\gamma-(x-\lambda)^\gamma\right]\ \ \mbox{for all}\ \ x\in(\lambda,\infty).
\end{equation}
Let $\varphi(x):=x^\gamma-(x-\lambda)^\gamma$, $x\in(\lambda,\infty)$. It is obvious that $\varphi$ is continuous on $(\lambda,\infty)$. Moreover, by letting $x=\frac{1}{t}$ and using L'Hospital's rule, we have that
\begin{align*}
\lim_{x\to\infty}\varphi(x)=\lim_{t\to0}\frac{1-(1-\lambda t)^\gamma}{t^\gamma}=\lim_{t\to0}\frac{\lambda(1-\lambda t)^{\gamma-1}}{t^{\gamma-1}}=\lambda\lim_{t\to0}\left(\frac{1}{t}-\lambda\right)^{\gamma-1}=0
\end{align*}
for $\gamma\in(0,1)$. This implies that for any $\mu>0$, there exists sufficient large $x\in\bR_+$ such that $\mu>\lambda^{-\gamma}\varphi(x)$, which contradicts \eqref{ineq2_absproxexam}. Thus $T$ is not GAN with exponent $\gamma$. According to Proposition \ref{prop_contrsubSN}, we know that $T$ is not FP-contractive. Naturally, it is not contractive either.
\end{proof}

We mention here that the projection operator $P_E$ onto a closed convex set $E\subset\bRn $ may not be GAN with exponent 1. In addition, neither the proximity operator of $\ell_1$ norm nor the proximity operator of $\ell_2$ norm is GAN with exponent 1 with respect to $\ell_2$ norm when the dimension is greater than 2. The $\ell_1$ norm is defined by $\|x\|_1:=\sum_{i=1}^n|x_i|$ for $x\in\bRn$.

We next provide a theorem showing that there exists a class of GAN operators with exponent 1 for a high-dimensional case. An example satisfies this theorem will be given later in Corollary \ref{LSexample}.

\begin{theorem}\label{thm_TdbLipSN1}
Let $T:\bRn\to\bRn$ be a firmly nonexpansive operator. If there exists $\alpha\in(0,1]$ such that 
\begin{equation}\label{InvLip}
\|Tx-Ty\|\geq\alpha\|x-y\|,\ \ \mbox{for all}\ \ x,y\in\bRn,
\end{equation}
then for $\beta\in\left(0,2\right)$, $\mI-\beta T$ is GAN with exponent 1.
\end{theorem}
\begin{proof}
It suffices to show that there exists $\mu\in\bR_+$ such that for all $x,y\in\bRn$,
\begin{equation}\label{TLipSNneq0}
\left\|(x-y)-\beta\left(Tx-Ty\right)\right\|+\mu\left\|\beta(Tx-Ty)\right\|\leq\|x-y\|.
\end{equation}
For any $\beta\in\left(0,2\right)$, we are able to find some sufficient small $\mu\in(0,\alpha)$ such that the following two inequalities hold:
\begin{equation}\label{TLipSNbetaneq1}
\beta\leq\frac{1}{\mu}\ \text{and}\ \beta\leq\frac{1}{1-\mu^2}\left(2-\frac{2\mu}{\alpha}\right).
\end{equation}
Let $w:=x-y$, $v:=Tx-Ty$. It follows from the second inequality of \eqref{TLipSNbetaneq1} that
$$
2-(1-\mu^2)\beta>0\ \text{and}\ \frac{2\mu}{2-(1-\mu^2)\beta}\leq\alpha,
$$
which together with \eqref{InvLip} imply that 
$$
\|v\|\geq\frac{2\mu}{2-(1-\mu^2)\beta}\|w\|,
$$
and hence
$$
\left(2\beta-(1-\mu^2)\beta^2\right)\|v\|^2\geq2\mu\beta\|w\|\|v\|.
$$
Further, by the firm nonexpansiveness of $T$, i.e., $\|v\|^2\leq\langle w,v\rangle$, we get that
$$
(1-\mu^2)\beta^2\|v\|^2+2\mu\beta\|w\|\|v\|\leq2\beta\langle w,v\rangle,
$$
which is equivalent to
\begin{equation}\label{TLipSNneq6}
\|w-\beta v\|^2\leq(\|w\|-\mu\|\beta v\|)^2.
\end{equation}
The nonexpansiveness of $T$ and the first inequality of \eqref{TLipSNbetaneq1} give that $\mu\|\beta v\|\leq\|w\|$. This combines with \eqref{TLipSNneq6} implies that
$$
\|w-\beta v\|\leq\|w\|-\mu\|\beta v\|,
$$
that is, \eqref{TLipSNneq0} holds. Therefore, $\mI-\beta T$ is GAN with exponent 1.
\end{proof}

Note that the identity operator is the trivial GAN operator with exponent $\gamma$ for any $\gamma\in\bR_+$. In the next proposition, we identify ranges of $\mu$ for the non-triviality of GAN operators for different ranges of $\gamma$. To simplify the notation, throughout the remaining part of this paper, we define
\begin{equation}\label{OmegamuSN}
\ymuSNset:=\{T\in\Lambda: \text{$T$ is $\mu$-GAN with exponent $\gamma$}\}.
\end{equation}

\begin{proposition}\label{prop_nontrivial}
Let $\gamma,\mu\in\bR_+$. 
\begin{itemize}
\item[$(i)$] For any $\gamma\leq1$,
$$
\ymuSNset\backslash\{\mI\}\neq\varnothing\ \ \mbox{if and only if}\ \ \mu\leq1.
$$
\item[$(ii)$] For any $\gamma>1$ and $\mu\in\bR_+$, $\ymuSNset\backslash\{\mI\}\neq\varnothing$.
\end{itemize}
\end{proposition}
\begin{proof}
We first establish $(i)$. Suppose that $\mu\leq1$ and show that $\ymuSNset\backslash\{\mI\}\neq\varnothing$. It suffices to find some operator $T\neq\mI$ such that $T\in\ymuSNset$ for any $\gamma\in(0,1]$. To this end, we define $T:\bRn\to\bRn$ by $T(x):=z$ for all $x\in\bRn$, where $z\in\bRn$ is a constant vector. Since $\mu\leq1$, for any $\gamma\in\bR_+$ and for all $x,y\in\bRn$, we have that
$$
\|Tx-Ty\|^\gamma+\mu\|(\mI-T)x-(\mI-T)y\|^\gamma=\mu\|x-y\|^\gamma\leq\|x-y\|^\gamma,
$$
that is, \eqref{neq_SNdef} holds. Hence $T\in\ymuSNset\backslash\{\mI\}$ for any $\gamma\in(0,1]$. 

Conversely, for any $\gamma\in(0,1]$, if $\ymuSNset\backslash\{\mI\}\neq\varnothing$, then there exists $T\in\ymuSNset$ such that for some $x\in\bRn$, $Tx\neq x$. Since $T\in\ymuSNset$, for any given $\hx\in\FixT$, we have that
\begin{equation}\label{nontrivSNineq2}
\|Tx-\hx\|^\gamma+\mu\|Tx-x\|^\gamma\leq\|x-\hx\|^\gamma.
\end{equation}
We next prove that the validity of \eqref{nontrivSNineq2} implies $\mu\leq1$. By \eqref{nontrivSNineq2} and the fact that $\|Tx-x\|>0$, we know that $\|Tx-\hx\|<\|x-\hx\|$. Let $\alpha:=\frac{\|Tx-\hx\|}{\|x-\hx\|}$. Then $\alpha\in[0,1)$ and
\begin{equation}\label{nontrivTxhxa}
\|Tx-\hx\|=\alpha\|x-\hx\|. 
\end{equation}
Hence
\begin{equation}\label{nontrimurangeeq1}
\|x-\hx\|^\gamma-\|Tx-\hx\|^\gamma=(1-\alpha^\gamma)\|x-\hx\|^\gamma
\end{equation}
and, by the triangle inequality,
\begin{equation}\label{nontrimurangeeq2}
\|Tx-x\|^\gamma\geq\left(\|x-\hx\|-\|Tx-\hx\|\right)^\gamma=(1-\alpha)^\gamma\|x-\hx\|^\gamma.
\end{equation}
By combing \eqref{nontrivSNineq2}, \eqref{nontrimurangeeq1} and \eqref{nontrimurangeeq2}, we obtain that
\begin{equation}\label{nontrimurangeeq3}
\mu\leq\frac{\|x-\hx\|^\gamma-\|Tx-\hx\|^\gamma}{\|Tx-x\|^\gamma}\leq\frac{1-\alpha^\gamma}{(1-\alpha)^\gamma}=\psi(\alpha),\ \ \alpha\in[0,1),
\end{equation}
where $\psi$ is defined by \eqref{psialpha}.
It follows from Lemma \ref{lemma_SNsubQC} that for any $\gamma\in(0,1)$,
\begin{equation}\label{maxpsialpha}
\max_{\alpha\in[0,1)}\left\{\psi(\alpha)\right\}=1.
\end{equation}
It is obvious that \eqref{maxpsialpha} also holds for $\gamma=1$. That is, \eqref{maxpsialpha} holds for any $\gamma\in(0,1]$, which together with \eqref{nontrimurangeeq3} implies that $\mu\leq1$.

Now we prove $(ii)$. For any $\gamma>1$, if $\mu\leq1$, we have shown in $(i)$ that there exists a constant operator $T$ such that $T\in\ymuSNset$. If $\mu>1$, we let $\psi$ be defined by  \eqref{psialpha}. Since $\gamma>1$, by Lemma \ref{lemma_SNsubQC}, there exists $\alpha\in(0,1)$ such that $\psi(\alpha)\geq\mu$, that is,
\begin{equation}\label{nontrivalphay}
\alpha^\gamma+\mu(1-\alpha)^\gamma\leq1.
\end{equation}
We next verify that $\alpha\mI\in\ymuSNset$. By using \eqref{nontrivalphay}, for all $x,y\in\bRn$, we have that
$$
\|\alpha x-\alpha y\|^\gamma+\mu\|(1-\alpha)x-(1-\alpha)y\|^\gamma=\left(\alpha^\gamma+\mu(1-\alpha)^\gamma\right)\|x-y\|^\gamma\leq\|x-y\|^\gamma.
$$
Thus, $\alpha\mI\in\ymuSNset\backslash\{\mI\}$.
\end{proof}

We know that the class of averaged nonexpansive operators is closed under the composition operation, that is, the composition of two averaged operators is still averaged nonexpansive \cite{bauschke2017convex}. This property is important for its application in convex optimization. In the following proposition, we prove the closeness of the class of GAN operators with exponent $\gamma$ under the composition operation for $\gamma\geq1$.

\begin{proposition}\label{prop_SNcompos}
Let $\gamma\in[1,+\infty)$ and $\mu_1, \mu_2\in\bR_+$. If $T_1\in\Omega_{\mu_1}^{\gamma}$ and $T_2\in\Omega_{\mu_2}^{\gamma}$, then $T_1\circ T_2\in\Omega_{\mu}^{\gamma}$, where $\mu:=2^{1-\gamma}\min\{\mu_1,\mu_2\}$.
\end{proposition}
\begin{proof}
For any $x,y\in\bRn$, define
$$
p:=(\mI-T_2)x-(\mI-T_2)y,\ \ q:=(\mI-T_1)(T_2x)-(\mI-T_1)(T_2y).
$$
Then, direct computation leads to
\begin{equation}\label{compSNpqeq}
p+q=(\mI-T_1\circ T_2)x-(\mI-T_1\circ T_2)y.
\end{equation}
Recall from Example 8.23 of \cite{bauschke2017convex} that
\begin{equation}\label{ineqexamp823}
\|p+q\|^\gamma\leq2^{\gamma-1}\left(\|p\|^\gamma+\|q\|^\gamma\right).
\end{equation}
Let $\mu:=2^{1-\gamma}\min\{\mu_1,\mu_2\}$. It follows from \eqref{compSNpqeq} and \eqref{ineqexamp823} that
\begin{equation}\label{ineqmupq}
\mu\|(\mI-T_1\circ T_2)x-(\mI-T_1\circ T_2)y\|^\gamma\leq\mu_2\|p\|^\gamma+\mu_1\|q\|^\gamma.
\end{equation}
By the fact that $T_1\in\Omega_{\mu_1}^{\gamma}$ and $T_2\in\Omega_{\mu_2}^{\gamma}$, we have that
$$
\mu_1\|q\|^\gamma\leq\|T_2x-T_2y\|^\gamma-\|(T_1\circ T_2)x-(T_1\circ T_2)y\|^\gamma
$$
and
$$
\mu_2\|p\|^\gamma\leq\|x-y\|^\gamma-\|T_2x-T_2y\|^\gamma.
$$
Adding the above two inequalities together yields
\begin{equation}\label{mu2qmu1p}
\mu_2\|p\|^\gamma+\mu_1\|q\|^\gamma\leq\|x-y\|^\gamma-\|(T_1\circ T_2)x-(T_1\circ T_2)y\|^\gamma.
\end{equation}
Now combining \eqref{ineqmupq} and \eqref{mu2qmu1p}, we conclude that
$$
\mu\|(\mI-T_1\circ T_2)x-(\mI-T_1\circ T_2)y\|^\gamma\leq\|x-y\|^\gamma-\|(T_1\circ T_2)x-(T_1\circ T_2)y\|^\gamma,
$$
which completes the proof.
\end{proof}

Before closing this section, we illustrate certain geometric properties of nonexpansive, firmly nonexpansive, averaged nonexpansive and contractive operators, and the proposed GAN operators with different exponents. Such geometric properties are useful in guiding us for the convergence analysis of the Picard sequences of these operators. For $x\in\bRn$, $r\in\bR_+$, we define the ball with center $x$ and radius $r$ by
$$
B(x,r):=\{y\in\bRn|\ \|y-x\|\leq r\}.
$$
Let $T\in\Lambda$ and $\hx$ be an arbitrary fixed-point of $T$. As shown in Fig. \ref{fig1}, for a given $x\in\bRn$, ranges of $Tx$ are illustrated by balls with distinct centers and radii for the cases for $T$ being nonexpansive, contractive, firmly nonexpansive and averaged nonexpansive. 

\begin{figure}[H]
\centering
\includegraphics[width=0.8\textwidth]{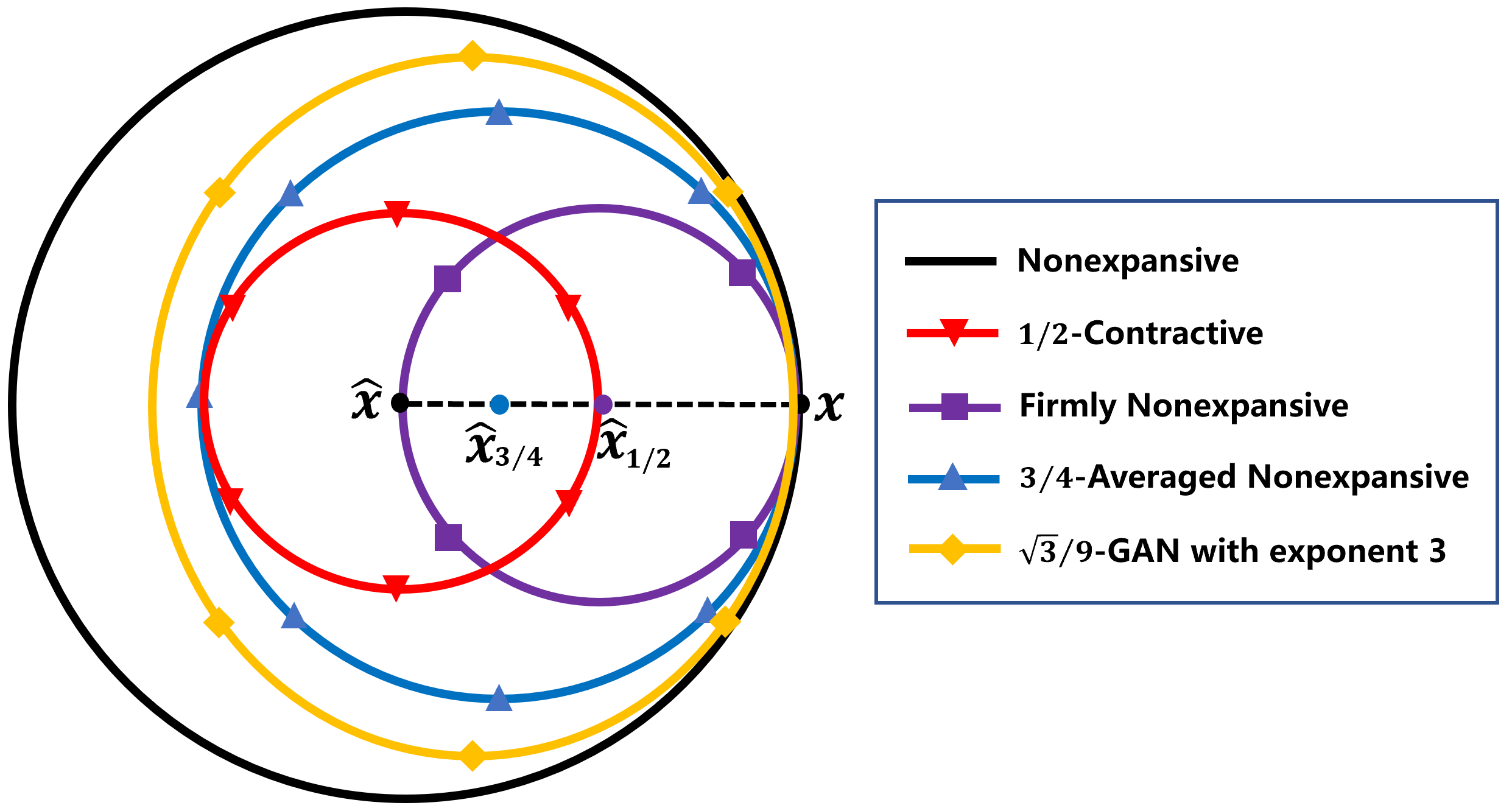}
\caption{The range of $Tx$ for a given $x\in\bR^2$ when $T$ is nonexpansive, contracitve, firmly nonexpansive, averaged nonexpansive or GAN with exponent 3 with respect to $\|\cdot\|_2$: inner region of the circles including the boundaries.}
\label{fig1}
\end{figure}

\begin{figure}[H]
\centering
\includegraphics[width=0.85\textwidth]{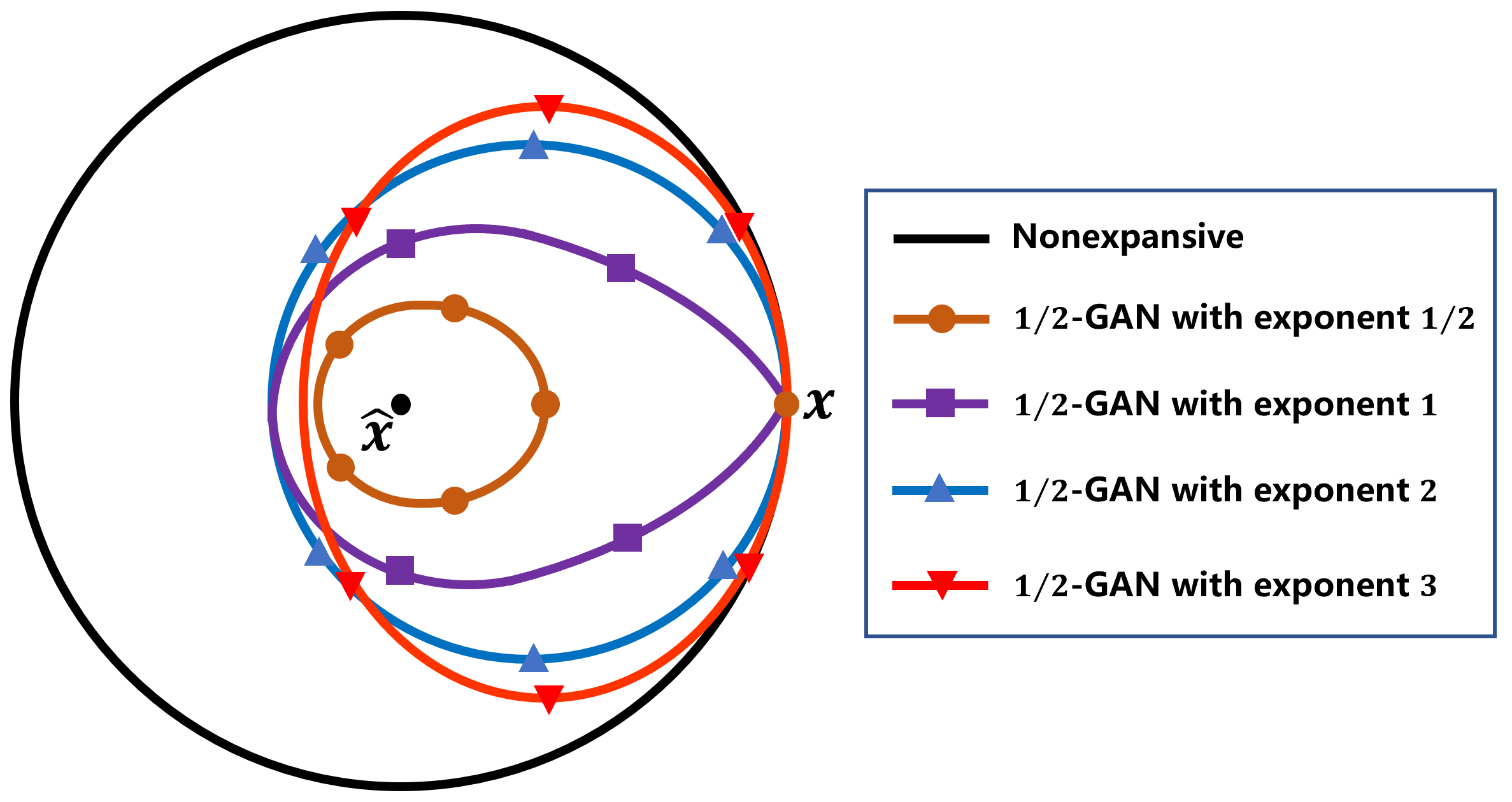}
\caption{The range of $Tx$ for a given $x\in\bR^2$ when $T$ is nonexpansive and $\frac{1}{2}$-GAN with exponent $\frac{1}{2}$, $1$, $2$ and $3$ with respect to $\|\cdot\|_2$: inner region of the closed curves including the boundaries.}
\label{fig2}
\end{figure}

From Fig. \ref{fig1}, we can see that if $T$ is nonexpansive, then $Tx$ may stay on the boundary of $B(\hx,\|x-\hx\|)$ and not be equal to $x$, that is, $x$ is not a fixed-point of $T$ and the distance from $Tx$ to the fixed-point $\hx$ remains the same as that from $x$ to $\hx$. In the same way, for any positive integer $k$, $T^kx$ may always stay on the boundary, which tells that the fixed-point iteration of $T$ may not converges when $T$ is nonexpansive. If $T$ is contractive, the range of $T^{k}x$ will shrink as $k$ increases, which leads to the convergence of $T^kx$ to $\hx$. For the case that $T$ is averaged nonexpansive, the range of $Tx$ is an inscribed ball of $B(\hx,\|x-\hx\|)$ with tangent point $x$, and convergence of the fixed-point iteration of $T$ in this case needs further study (see Theorem \ref{thm_SNconv} for more details).

In Fig. \ref{fig2}, we show the range of $Tx$ for the case that $T$ is $\frac{1}{2}$-GAN with exponent $\frac{1}{2}$, $1$, $2$ or $3$. When the exponent $\gamma$ is equal to 2, the range of $Tx$ is a ball the same as in the case of averaged nonexpansiveness. It is of the egg shape for exponent $3$ and the water-drop shape for exponent 1. We will show that the convergence rate of the fixed-point iteration of $T$ improves as $\gamma$ decreases. Especially, when $\gamma<1$, the range of $T$ has some kind of contractive property (the point $x$ is included), which is called as the FP-contractive property.

\section{Local convergence rate analysis}\label{seclocrate}

In this section, we establish convergence of the fixed-point iteration of a GAN operator. The local convergence rate (the convergence rate of the distance between two consecutive iterates) is also analyzed. We will show that the local convergence rate of the fixed-point iteration of a GAN operator with exponent $\gamma$ is $o(k^{-\frac{1}{\gamma}})$. The smaller the exponent $\gamma$ a GAN operator has, the higher local convergence rate its fixed-point iteration results.

We first describe the fixed-point iteration of an operator. By $\bN_0$ and $\bN_+$ we denote the set of all nonnegative integers and the set of all positive integers, respectively. Given an initial vector $x^{0}\in\bRn$, the fixed-point iteration of $T:\bRn\to\bRn$ is given by
$$
x^{k+1}=Tx^k, \ \ k\in \bN_0.
$$
We call the sequence $\{x^k\}$ generated by the fixed-point iteration of $T$ a Picard sequence of operator $T$.

We begin with stating the main theorem of this section.  For two sequences $\{a_k\}\subset\bR_+\cup\{0\}$ and $\{b_k\}\subset\bR_+$, both tending to zero, if $\lim_{k\to\infty}\frac{a_k}{b_k}=0$, we write $a_k=o(b_k)$. If there exist constants $c>0$ and $K\in\bN_0$ such that $a_k\leq cb_k$ for all $k\geq K$, we write $a_k=O(b_k)$.

\begin{theorem}\label{thm_SNconv}
If $T\in\Lambda$ is GAN with exponent $\gamma\in\bR_+$, then for any initial vector $x^0\in\bRn$, the Picard sequence $\{x^k\}$ of operator $T$ converges to some $x^*\in\FixT$, and
\begin{equation}\label{eq_SNconvrate}
\|x^{k+1}-x^k\|=o\left(k^{-\frac{1}{\gamma}}\right).
\end{equation}
\end{theorem}

We now proceed to prove Theorem \ref{thm_SNconv}. To this end, we recall Proposition 5.28 of \cite{bauschke2017convex} as a lemma. 

\begin{lemma}\label{lemma_genFPconv}
Let $T\in\Lambda$ be a nonexpansive operator. For any initial vector $x^0\in\bRn$, if the Picard sequence $\{x^k\}$ of $T$ satisfies that $\lim_{k\to\infty}\|x^{k+1}-x^{k}\|=0$, then $\{x^k\}$ converges to a fixed-point of $T$.
\end{lemma}

We also need Lemma 3 of \cite{davis2016convergence}, which we state below.

\begin{lemma}\label{lemma_SNconv}
Suppose that $\{a_k\}$ and $\{b_k\}$ are two nonnegative sequences in $\bR$. If $\sum_{k=0}^{\infty}a_kb_k<\infty$, $\{b_k\}$ is monotonically decreasing, and there exists $\varepsilon>0$ such that $a_k\geq\varepsilon$ for all $k\in\bN_0$, then $b_k=o\left(\frac{1}{k}\right)$.
\end{lemma}


\begin{proof}[Proof of Theorem \ref{thm_SNconv}]
We first show convergence of the sequence $\{x^k\}$. Since $T$ is GAN with exponent $\gamma$, we know that it is nonexpansive and there exists $\mu\in\bR_+$ such that for any $\hx\in\FixT$,
\begin{equation}\label{SNconv_neq1}
\|x^{k+1}-\hx\|^\gamma+\mu\|x^{k+1}-x^k\|^\gamma\leq\|x^k-\hx\|^\gamma.
\end{equation}
For any positive integer $K$, summing both sides of the inequality \eqref{SNconv_neq1} for $k=0,1,\ldots,K$ yields that
\begin{equation}\label{SNconv_neq2}
\sum_{k=0}^{K}\mu\|x^{k+1}-x^k\|^\gamma\leq\|x^0-\hx\|^\gamma-\|x^{K+1}-\hat{x}\|^\gamma\leq\|x^0-\hat{x}\|^\gamma.
\end{equation}
Inequality \eqref{SNconv_neq2} ensures that series 
\begin{equation}\label{ConVS}
\sum_{k=0}^\infty\mu\|x^{k+1}-x^k\|^\gamma<\infty.
\end{equation}
Result \eqref{ConVS} implies that $\lim_{k\to\infty}\|x^{k+1}-x^k\|=0$. By Lemma \ref{lemma_genFPconv}, we conclude that $\{x^k\}_{k\in\bN_0}$ converges to some $x^*\in\FixT$.

We next employ Lemma \ref{lemma_SNconv} to show that \eqref{eq_SNconvrate} holds. 
Applying Lemma \ref{lemma_SNconv} to the sequences  $a_k:=\mu$ and $b_k:=\|x^{k+1}-x^k\|^\gamma$, $k\in\bN_0$, it suffices to show that $\{\|x^{k+1}-x^k\|\}$ is monotonically decreasing. This follows from the nonexpansiveness of $T$ since it implies that
$$
\|x^{k+2}-x^{k+1}\|=\|Tx^{k+1}-Tx^k\|\leq\|x^{k+1}-x^k\|
$$
for all $k\in\bN_0$. Therefore,  by Lemma \ref{lemma_SNconv}, \eqref{eq_SNconvrate} holds.
\end{proof}

Since an averaged nonexpansive operator is GAN with exponent $\gamma=2$, Theorem \ref{thm_SNconv} covers the well-known result that the local convergence rate of the fixed-point iteration of an averaged nonexpansive operator is $o(k^{-\frac{1}{2}})$, see \cite{baillion1996rate}. Moreover, it ensures that the local convergence rate of the fixed-point iteration of a GAN operator with exponent $\gamma<2$ is higher than that of an averaged nonexpansive operator.

\section{Global convergence rate analysis}\label{secglobrate}
%
We consider in this section the global convergence rate (the convergence rate in terms of the distance between an iterate and a fixed-point) of the fixed-point iteration of GAN operator and investigate the relation between the local convergence rate and the global convergence rate. We will show that the fixed-point iteration of a GAN operator with exponent $\gamma\in(0,1)$ can achieve an exponential global convergence rate. Moreover, if a GAN operator is also H$\ddot{\text{o}}$lder regular, the global convergence rate of its fixed-point iteration will depend on both the exponent of generalized averaged nonexpansiveness and the exponent of H$\ddot{\text{o}}$lder regularity. The definition of H$\ddot{\text{o}}$lder regularity will be given later in this section.

We first establish a relation between the local convergence rate and the global convergence rate.

\begin{theorem}\label{thm_rellocglob}
If $T\in\Omega_{\mu}^{1}$ for $\mu\in(0,1]$, then for any initial vector $x^0\in\bRn$, the Picard sequence $\{x^k\}$ of $T$ converges to some $x^*\in\FixT$, and there holds the equivalence relation for all positive integers $k$,
\begin{equation}\label{equivDCIDIS1}
\mu\sum_{j=k}^\infty\|x^{j+1}-x^{j}\|\leq\|x^k-x^*\|\leq\sum_{j=k}^\infty\|x^{j+1}-x^{j}\|.
\end{equation}
\end{theorem}
\begin{proof}
Convergence of $\{x^k\}$ to some $x^*\in\FixT$ follows from Theorem \ref{thm_SNconv}. 

It remains to establish the equivalence relation \eqref{equivDCIDIS1}. Since $T\in\Omega_{\mu}^{1}$, by the definition of generalized averaged nonexpansiveness, we have that
$$
\mu\|x^{j+1}-x^{j}\|\leq\|x^j-x^*\|-\|x^{j+1}-x^*\|,\ \ j\in\bN_0.
$$
For any $N>k$, summing the above inequality for $j=k,k+1,\ldots,N$ yields that
\begin{equation}\label{UUper-bound}
    \mu\sum_{j=k}^N\|x^{j+1}-x^{j}\|\leq\|x^k-x^*\|-\|x^{N+1}-x^*\|\leq\|x^k-x^*\|.
\end{equation}
In the inequality above, we let $N\to \infty$ and get the left inequality of \eqref{equivDCIDIS1}.

To establish the right inequality of \eqref{equivDCIDIS1}, for any $N>k$, we repeatedly use the triangle inequality and obtain that
\begin{equation}\label{RUper-bound}
\|x^k-x^*\|\leq\sum_{j=k}^N\|x^j-x^{j+1}\|+\|x^{N+1}-x^*\|.
\end{equation}
Inequality \eqref{UUper-bound} implies that 
$$
\sum_{j=k}^\infty\|x^{j+1}-x^{j}\|<\infty.
$$
Moreover, the first part of this theorem ensures that 
$$
\lim_{N\to\infty}\|x^{N+1}-x^*\|=0.
$$
Hence,
letting $N\to \infty$ in inequality \eqref{RUper-bound} yields the right inequality of \eqref{equivDCIDIS1}.
\end{proof}

Theorem \ref{thm_rellocglob} indicates that when the operator $T$ is GAN with exponent 1, the global convergence rate of its Picard sequence is equivalent to the convergence rate of $\sum_{j=k}^\infty\|x^{j+1}-x^{j}\|$. We next show how Theorem \ref{thm_rellocglob} provides a way to estimate the global convergence rate. We first show a technical result.

\begin{proposition}\label{prop_ratemin1}
If $\{a_k\}\subset\bR$ is a nonnegative sequence with $a_k=o\left(k^{-\alpha}\right)$, then $\sum_{j=k}^{\infty}a_j=o\left(k^{-(\alpha-1)}\right)$.
\end{proposition}
\begin{proof}
Since $a_k=o\left(k^{-\alpha}\right)$, for any $\varepsilon>0$, there is $K\in\bN_0$ such that $a_j<\frac{\varepsilon}{j^\alpha}$ for all $j\geq K$. Summing this inequality for $j=k, k+1, \dots,$ with $k\geq K$, we obtain that
$$
\sum_{j=k}^{\infty}a_j< \varepsilon\sum_{j=k}^{\infty}\frac{1}{j^{\alpha}}\leq \varepsilon\int_{k-1}^\infty\frac{1}{t^{\alpha}}dt=\frac{\varepsilon}{(\alpha-1)(k-1)^{\alpha-1}}.
$$
This establishes the desired estimate.
\end{proof}

Theorem \ref{thm_rellocglob} together with Propositions \ref{prop_SNinclrel} and \ref{prop_ratemin1} leads to the following theorem.

\begin{theorem}\label{convrateorder1}
If $T\in\Lambda$ is GAN with exponent $\gamma\in(0,1)$, then for any initial vector $x^0\in\bRn$, the Picard sequence $\{x^k\}$ of $T$ converges to some $x^*\in\FixT$, and $\|x^k-x^*\|=o\left(k^{-\frac{1-\gamma}{\gamma}}\right)$.
\end{theorem}
\begin{proof}
It follows from Theorem \ref{thm_SNconv} that $\{x^k\}$ convergence to some $x^*\in\FixT$ and $\|x^{k+1}-x^k\|=o\left(k^{-\frac{1}{\gamma}}\right)$. Applying  Proposition \ref{prop_ratemin1} with $a_k:=\|x^{k+1}-x^k\|$, we obtain that 
$$
\sum_{j=k}^{\infty}\|x^{j+1}-x^j\|=o\left(k^{-\frac{1-\gamma}{\gamma}}\right).
$$
Moreover, by Proposition \ref{prop_SNinclrel}, we see that $T$ is GAN with exponent 1. Thus, the desired result of this theorem follows from Theorem \ref{thm_rellocglob}.
\end{proof}

In fact, according to Proposition \ref{prop_SNsubQC}, we know that GAN operator with exponent $\gamma\in(0,1)$ is FP-$\rho$-contractive for some $\rho\in(0,1)$, which leads to higher order global convergence rate of its Picard sequence than the result shown in Theorem \ref{convrateorder1}. To this end, we first show that the Picard sequence of a FP-contractive operator has exponential global convergence rate.

\begin{theorem}\label{lemma_QCconv}
If operator $T\in\Lambda$ is FP-contractive, then for any initial vector $x^0\in\bRn$, the Picard sequence $\{x^k\}$ of $T$ either converges to some $x^*\in\FixT$ within a finite number of iterations or there exists $\rho\in(0,1)$ such that 
\begin{equation}\label{GCONVR}
    \|x^k-x^*\|\leq\rho^k\|x^0-x^*\|,\ \ \mbox{for all}\ \ k\in\bN_0.
\end{equation}
\end{theorem}
\begin{proof}
If there exists an integer $K\in\bN_0$ such that $x^K\in\FixT$, then $x^k=x^K$ for all $k>K$, and hence $\lim_{k\to\infty}x^k=x^K$. Otherwise, $x^k\notin\FixT$ for all $k\in\bN_0$. In this case, by the definition of the FP-contractive operator, there exist $x^*\in\FixT$ and $\rho\in(0,1)$ such that 
$$
\|x^{k+1}-x^*\|\leq\rho\|x^k-x^*\|,\ \  \mbox{for all}\ \ k\in\bN_0.
$$
Repeatedly using this inequality, we obtain the desired estimate \eqref{GCONVR}.
\end{proof}

The next corollary improves the global convergence rate given in Theorem \ref{convrateorder1}.
\begin{corollary}\label{thm_SNconv1}
If operator $T\in\Lambda$ is GAN with exponent $\gamma\in(0,1)$, then for any initial vector $x^0\in\bRn$, the Picard sequence $\{x^k\}$ of $T$ either converges to some $x^*\in\FixT$ within finite iterations or there exists some $\rho\in(0,1)$ such that estimate \eqref{GCONVR} holds.
\end{corollary}

\begin{proof}
By Proposition \ref{prop_SNsubQC}, a GAN operator $T\in\Lambda$ with exponent $\gamma\in(0,1)$ is FP-contractive. Then the desired result of this corollary follows from Theorem \ref{lemma_QCconv}.
\end{proof}

To obtain global convergence rates for the case with the exponent $\gamma\geq1$, we need an additional condition that establishes a relation between the local convergence rate and the global convergence rate. In view of this, we recall the definition of H$\ddot{\text{o}}$lder regular operators introduced in \cite{borwein2017convergence}. For a set $E\subset\bRn$ and $x\in \bRn$, we define 
$$
d(x,E):=\inf_{y\in E}\{\|x-y\|\}.
$$

\begin{definition}
Let $T\in\Lambda$. We say that $T$ is a H$\ddot{\text{o}}$lder regular (HR) operator with exponent $\gamma$, if there exist $\gamma\in\bR_+$ and $\mu\in\bR_+$ such that
$$
d(x,\FixT)\leq\mu\|x-Tx\|^\gamma,\ \ \mbox{for all}\ \ x\in\bRn.
$$
\end{definition}

We verify below that for any $\rho\in(0,1)$, a FP-$\rho$-contractive operator $T\in\Lambda$ is HR with exponent 1. By the FP-contractivity of $T$ and the the triangle inequality, for all $x\in\bRn$ and $\hx\in\FixT$, we have that
$$
\|Tx-x^*\|\leq\rho\|x-x^*\|\ \ \mbox{and}\ \  \|x-x^*\|\leq\|Tx-x^*\|+\|x-Tx\|,
$$
which imply that $\|x-x^*\|\leq\frac{1}{1-\rho}\|x-Tx\|$, and hence
$$
d(x,\FixT)\leq\frac{1}{1-\rho}\|x-Tx\|\ \ \mbox{for all}\ \ x\in\bRn.
$$
Thus, $T$ is HR with exponent 1. We shall show in the next section that the gradient descent operator is also HR with exponent 1 under appropriate assumptions.

Now we state the main result on the global convergence rate of the fixed-point iteration of GAN operators.

\begin{theorem}\label{thm_SNandBHR}
If operator $T\in\Lambda$ is GAN with exponent $\gamma_1\in\bR_+$ and HR with exponent $\gamma_2\in\bR_+$, then for any initial vector $x^0\in\bRn$, the Picard sequence $\{x^k\}$ of $T$ converges to some $x^*\in\FixT$, and there exists $\rho\in(0,1)$ such that
\begin{equation}\label{eqSNBHRrate}
\|x^k-x^*\|=\begin{cases}
O\left(k^{-\frac{\gamma_2}{\gamma_1(1-\gamma_2)}}\right),&0<\gamma_2<1,\\
O\left(\rho^k\right),&\gamma_2\geq1.
\end{cases}
\end{equation}
\end{theorem}

To prove Theorem \ref{thm_SNandBHR}, we recall Lemma 4.1 of \cite{borwein2014analysis}.

\begin{lemma}\label{recurrelatlem}
Suppose that $\{a_k\}$ and $\{b_k\}$ be two sequences of nonnegative numbers. For $p>0$, if there exists $K\in\bN_0$ such that
$$
a_{k+1}\leq a_k(1-b_ka_k^p),\ \ \mbox{for all}\ \ k\geq K,
$$
then
$$
a_k\leq\left(a_K^{-p}+p\sum_{j=K}^{k-1}b_j\right)^{-\frac{1}{p}}, \ \ \mbox{for all}\ \ k>K.
$$
\end{lemma}

For a closed and convex set $E\subset\bRn$, we define $P_{E}(x):=\argmin_{y\in E}\{\|x-y\|\}$. Note that $\FixT$ is closed and convex if $T\in\Lambda$ is nonexpansive. Hence, $P_{\FixT}(x)$ is well-defined, which will be used in the proof of the next Proposition.

\begin{proposition}\label{propforBHR}
Suppose that $T\in\Lambda$ is nonexpansive. For the Picard sequence $\{x^k\}$ of $T$ with a given initial vector $x^0\in\bRn$,  let $d_k:=d(x^k,\FixT)$, $k\in\bN_0$. If there exist $\gamma>0$, $\mu>0$, $\vartheta\geq1$ and $K\in\bN_0$ such that
\begin{equation}\label{ineq_BHRlemma}
d_{k+1}^\gamma\leq d_k^\gamma-\mu d_k^{\gamma\vartheta},\ \ \mbox{for all}\ \  k\geq K,
\end{equation} 
then $\{x^k\}$ converges to some $x^*\in\FixT$. Moreover, there exist $C\in\bR_+$ and $\rho\in[0,1)$ such that for $k>K$,
$$
\|x^k-x^*\|\leq\begin{cases}
Ck^{-\frac{1}{\gamma(\vartheta-1)}},&\vartheta>1,\\
C\rho^{k-K},&\vartheta=1.
\end{cases}
$$
\end{proposition}

\begin{proof}
Let $a_k=d_k^\gamma$ and $p=\vartheta-1\geq0$. Then \eqref{ineq_BHRlemma} becomes
\begin{equation}\label{ineq_dkak}
a_{k+1}\leq a_{k}(1-\mu a_{k}^p),\ \ \mbox{for all}\ \  k\geq K.  
\end{equation}
We consider two cases based on the value of $\vartheta$.

$Case\ 1$: $\vartheta>1$. We first show that $\{x^k\}$ converges to some $x^*\in\FixT$. It follows from Lemma \ref{recurrelatlem} with $b_k:=\mu$ that
$$
a_k\leq\left(a_K^{-p}+p\mu(k-K)\right)^{-\frac{1}{\vartheta-1}},\ \ \mbox{for all}\ \  k>K.
$$
Hence, there exists $C_1>0$ such that for $k>K$,
$$
d_k=a_k^{\frac{1}{\gamma}}\leq C_1k^{-\frac{1}{\gamma(\vartheta-1)}}\to 0, \ \ \mbox{as}\ \ k\to \infty.
$$
By the nonexpansiveness of $T$, we know that $\{\|x^{k}-\hx\|\}$ is monotonically decreasing for any $\hx\in\FixT$. Then
\begin{align*}
\|x^{k+1}-x^k\|&\leq\|x^{k+1}-P_{\FixT}(x^k)\|+\|x^{k}-P_{\FixT}(x^k)\| \\
&\leq2\|x^{k}-P_{\FixT}(x^k)\|=2d_k\to0.
\end{align*}
We conclude from Lemma \ref{lemma_genFPconv} that $\{x^k\}$ converges to some $x^*\in\FixT$. Using the monotonicity of $\{\|x^{k}-\hx\|\}$ for any $\hx\in\FixT$ again, we have that
$$
\|x^m-P_{\FixT}(x^k)\|\leq\|x^{m-1}-P_{\FixT}(x^k)\|\leq\cdots\leq\|x^{k}-P_{\FixT}(x^k)\|=d_k
$$
for all $m>k$, $k\in\bN_0$. Letting $m$ tend to infinity, the above inequality becomes
$$
\|x^*-P_{\FixT}(x^k)\|\leq d_k,\ \ \mbox{for all}\ \ k\in\bN_0,
$$
which together with the triangle inequality implies for all $k>K$ that
\begin{align*}
\|x^k-x^*\|&\leq\|x^k-P_{\FixT}(x^k)\|+\|x^*-P_{\FixT}(x^k)\|\\
&\leq2d_k\leq2C_1k^{-\frac{1}{\gamma(\vartheta-1)}}. 
\end{align*}

$Case\ 2$: Suppose that $\vartheta=1$. Then \eqref{ineq_BHRlemma} becomes $d_{k+1}^\gamma\leq(1-\mu)d_k^\gamma$ for all $k\geq K$. This implies that $\mu\in(0,1]$ and for $k>K$, $d_{k}\leq d_K(1-\mu)^\frac{k-K}{\gamma}\to0$. By the same argument as Case 1, there exists some $x^*\in\FixT$ such that
$$
\|x^k-x^*\|\leq2d_k\leq2d_K(1-\mu)^\frac{k-K}{\gamma}.
$$
Therefore, the proof is completed by setting $C=\max\{2C_1,2d_K\}$ and $\rho=(1-\mu)^\frac{1}{\gamma}$.
\end{proof}

Note that the result in Proposition 3.1 of \cite{borwein2017convergence} is a special case of the above proposition with $\gamma=2$. The generalization for any $\gamma\in\bR_+$ is necessary for the global convergence rate analysis of the fixed-iteration of GAN operator. We next employ Proposition \ref{propforBHR} to prove Theorem \ref{thm_SNandBHR}.

\begin{proof}[Proof of Theorem \ref{thm_SNandBHR}]
Since $T$ is GAN with exponent $\gamma_1$, by Theorem \ref{thm_SNconv}, we know that $\{x^k\}$ converges to some $x^*\in\FixT$ and $\lim_{k\to\infty}\|x^{k+1}-x^{k}\|=0$. Moreover, there exists $\mu_1\in\bR_+$ such that for all $k\in\bN_0$,
\begin{equation}\label{SNineq4BHR}
\|x^{k+1}-P_{\FixT}(x^k)\|^{\gamma_1}\leq\|x^k-P_{\FixT}(x^k)\|^{\gamma_1}-\mu_1\|x^{k+1}-x^{k}\|^{\gamma_1}.
\end{equation}
Let $d_k:=d(x^k,\FixT)$, $k\in\bN_0$. By the definition of $d_{k+1}$ and \eqref{SNineq4BHR}, we obtain that
\begin{equation}\label{BHRthmineq1}
d_{k+1}^{\gamma_1}\leq\|x^{k+1}-P_{\FixT}(x^k)\|^{\gamma_1}\leq d_k^{\gamma_1}-\mu_1\|x^{k+1}-x^{k}\|^{\gamma_1},\ \ \mbox{for all}\ \ k\in\bN_0.
\end{equation}
It follows from the H$\ddot{\text{o}}$lder regularity of $T$ that there exists $\mu_2\in\bR_+$ such that
\begin{equation}\label{BHRdk}
d_k\leq\mu_2\|x^{k+1}-x^{k}\|^{\gamma_2},\ \ \mbox{for all}\ \ k\in\bN_0.
\end{equation}
Since $\lim_{k\to\infty}\|x^{k+1}-x^{k}\|=0$, there exists $K$ such that $\|x^{k+1}-x^{k}\|<1$ for all $k\geq K$, which together with \eqref{BHRdk} implies that for $\gamma_2\geq1$,
\begin{equation}\label{BHRdk2}
d_k\leq\mu_2\|x^{k+1}-x^{k}\|,\ \ \mbox{for all}\ \ k\geq K.
\end{equation}
Now combing \eqref{BHRthmineq1} with \eqref{BHRdk} for $0<\gamma_2<1$ gives that
\begin{equation}\label{BHRdk3}
d_{k+1}^{\gamma_1}\leq d_k^{\gamma_1}-\mu_1\mu_2^{-\frac{\gamma_1}{\gamma_2}}d_k^{\frac{\gamma_1}{\gamma_2}},\ \ \mbox{for all}\ \ k\in\bN_0.
\end{equation}
Combing \eqref{BHRthmineq1} with \eqref{BHRdk2} for $\gamma_2\geq1$ gives that
\begin{equation}\label{BHRdk4}
d_{k+1}^{\gamma_1}\leq d_k^{\gamma_1}-\mu_1\mu_2^{-\gamma_1}d_k^{\gamma_1},\ \ \mbox{for all}\ \ k\geq K.
\end{equation}
Then we conclude from Proposition \ref{propforBHR} that there exist $C\in\bR_+$ and $\rho\in[0,1)$ such that for all $k>K$,
\begin{equation}\label{neqSNBHRrate}
\|x^k-x^*\|\leq\begin{cases}
Ck^{-\frac{\gamma_2}{\gamma_1(1-\gamma_2)}},&0<\gamma_2<1,\\
C\rho^{k-K},&\gamma_2\geq1.
\end{cases}
\end{equation}
For $0<\gamma_2<1$, $\|x^k-x^*\|=O\left(k^{-\frac{\gamma_2}{\gamma_1(1-\gamma_2)}}\right)$ follows from \eqref{neqSNBHRrate} immediately. We next consider the case $\gamma_2\geq1$. In this case, if $\rho=0$, then it is obvious that $\|x^k-x^*\|=O\left(\rho^{k}\right)$ holds according to \eqref{neqSNBHRrate}. If $\rho\in(0,1)$, then for all $k>K$, $\|x^k-x^*\|\leq C'\rho^{k}$, where $C'=C\rho^{-K}$. Thus $\|x^k-x^*\|=O\left(\rho^{k}\right)$.
\end{proof}

Theorem \ref{thm_SNandBHR} extends the result given in \cite{borwein2017convergence} where operators that are averaged nonexpansive (GAN with exponent $\gamma_1=2$) and HR with exponent $\gamma_2\in(0,1]$ were considered. 

We close this section by listing convergence rates of the fixed-point iterations of GAN operators with different exponents.

\begin{table}[H]\label{tab_convrate}
\caption{Convergence rates of the fixed-point iterations of GAN operators}
\centering
\begin{tabular}{c|l|l}
\toprule
Case &\hspace{5em}Conditions &Convergence rate\\
\hline
1& GAN with exponent $\gamma\in[1,\infty)$ & local: $o\left(k^{-\frac{1}{\gamma}}\right)$\\
\hline
2& GAN with exponent $\gamma\in(0,1)$ & global: exponential\\
\hline
\multirow{2}*{3}& GAN with exponent $\gamma_1\in[1,\infty)$ & \multirow{2}*{global: $O\left(k^{-\frac{\gamma_2}{\gamma_1(1-\gamma_2)}}\right)$}\\
&\& HR with exponent $\gamma_2\in(0,1)$ &\\
\hline
\multirow{2}*{4}& GAN with exponent $\gamma_1\in[1,\infty)$ &\multirow{2}*{global: exponential}\\
&\& HR with exponent $\gamma_2\in[1,\infty)$ &\\
\bottomrule
\end{tabular}
\end{table}

\section{Convergence rate analysis for optimization}\label{sec_convrateopti}
In this section, we first describe the fixed-point algorithms for the convex optimization problems described in Section \ref{sectFPopt}, and then employ the results in Sections \ref{seclocrate} and \ref{secglobrate} to analyze their convergence rates. The GAN operators provide a unified framework for developing fixed-point iterative schemes for convex optimization problems and analyzing their convergence and convergence rates.

By the definition \eqref{defT1} of operator $T_1$, its fixed-point iteration is the gradient descent algorithm given as follows:
\begin{equation}\label{algo1}
x^{k+1}=x^{k}-\beta\nabla f(x^k),\ \ \mbox{where}\ \ \beta\in\bR_+.
\end{equation}
The fixed-point iteration \eqref{defT2} of $T_2$ is given by
\begin{equation}\label{algo2}
x^{k+1}=\prox_{\beta g}(x^{k}-\beta\nabla f(x^k)),\ \ \mbox{where}\ \ \beta\in\bR_+.
\end{equation}
We next derive the fixed-point iteration of $T_3$ defined by \eqref{defT3}. Note that $W=R^{-1}G$, where $R:=\left(\begin{array}{cc}
\beta I_n&\\
&\eta I_m
\end{array}\right)$. We can verify that the fixed-point iteration $x^{k+1}=T_3(x^k)$ is equivalent to
$$
v^{k+1}=\widetilde{T}\left((E-G)v^{k+1}+(G-R\nabla r)(v^k)\right),
$$
that is,
$$
\begin{cases}
x^{k+1}=\prox_{\beta h}\left(x^{k}-\beta(\nabla f(x^k)+B^\top y^k)\right),\\
y^{k+1}=\prox_{\eta g^*}\left(y^k+\eta B(2x^{k+1}-x^k)\right).
\end{cases}
$$
By using the well-known Moreau decomposition \cite{lin2019krasnoselskii,moreau1962fonctions}
$$
\mI=\prox_{\eta g^*}+(\eta\mI)\circ\prox_{\frac{1}{\eta}g}\circ(\eta^{-1}\mI),
$$
we have the following fixed-point iteration of $T_3$,
\begin{equation}\label{algo3}
\begin{cases}
x^{k+1}=\prox_{\beta h}\left(x^{k}-\beta(\nabla f(x^k)+B^\top y^k)\right),\\
y^{k+1}=\eta\left(\mI-\prox_{\frac{1}{\eta}g}\right)\left(\frac{1}{\eta}y^k+B(2x^{k+1}-x^k)\right),
\end{cases}\ \mbox{where}\ \ \beta,\eta\in\bR_+.
\end{equation}

We next show the generalized averaged nonexpansiveness with exponent 2 of $T_1$, $T_2$ and $T_3$, which offers $o\left(k^{-\frac{1}{2}}\right)$ local convergence rate for algorithms \eqref{algo1}, \eqref{algo2} and \eqref{algo3}. We then provide higher order convergence rates for the fixed-point algorithms \eqref{algo1} and \eqref{algo2} under additional assumptions.

\begin{proposition}\label{ToptSN2}
Let $T_1$, $T_2$ and $T_3$ be defined by \eqref{defT1}, \eqref{defT2} and \eqref{defT3}, respectively. If $\beta<\frac{2}{L}$ for $T_1$, $T_2$ and $T_3$, and $\mu<\frac{2(2-\beta L)}{4\beta\|B\|_2^2+L(2-\beta L)}$ for $T_3$, then $T_1$ and $T_2$ are GAN with exponent 2 with respect to $\|\cdot\|_2$, $T_3$ is GAN with exponent 2 with respect to $\|\cdot\|_W$.
\end{proposition}

\begin{proof}
We first show the generalized averaged nonexpansiveness of $T_1$ and $T_2$. It follows from the proof of Theorem 26.14 of \cite{bauschke2017convex} that $T_1$ and $T_2$ are both averaged nonexpansive with respect to $\|\cdot\|_2$. Hence they are both GAN with exponent 2 with respect to $\|\cdot\|_2$ by Proposition \ref{prop_averaged}.

We now turn to considering operator $T_3$. It follows from Lemma 7 of \cite{lin2019krasnoselskii} that $T_3$ is averaged nonexpansive with respect to $\|\cdot\|_W$ if the minimum eigenvalue of $W$ is greater than $\frac{L}{2}$, that is, $W-\frac{L}{2}I$ is positive definite. Let $\tilde{B}:=\frac{1}{\sqrt{\left(\frac{1}{\beta}-\frac{L}{2}\right)\left(\frac{1}{\mu}-\frac{L}{2}\right)}}B$. According to Lemma 6.2 of \cite{li2015multi}, $W-\frac{L}{2}I$ is positive definite if and only if $\|\tilde{B}\|_2<1$, that is,
\begin{equation}\label{betamuB2}
\left(\frac{1}{\beta}-\frac{L}{2}\right)\left(\frac{1}{\mu}-\frac{L}{2}\right)>\|B\|_2^2.
\end{equation}
Since $\beta\in\left(0,\frac{2}{L}\right)$, $\eta\in\left(0,\frac{2(2-\beta L)}{4\beta\|B\|_2^2+L(2-\beta L)}\right)$, it is easy to verify that \eqref{betamuB2} holds, which implies that $T_3$ is averaged nonexpansive with respect to $\|\cdot\|_W$, and hence it is GAN with exponent 2 with respect to $\|\cdot\|_W$.
\end{proof}

\begin{proposition}\label{thm_convT123}
Suppose that $\beta<\frac{2}{L}$ and $\mu<\frac{2(2-\beta L)}{4\beta\|B\|_2^2+L(2-\beta L)}$. Then for arbitrary initial vectors $x^{0}\in\bRn$ and $y^{0}\in\bRm$, the following statements hold:
\begin{itemize}
\item[$(i)$] Sequence $\{x^k\}$ generated by Algorithm \eqref{algo1} converges to a minimizer of the objective function $f$.
\item[$(ii)$] Sequence $\{x^k\}$ generated by Algorithm \eqref{algo2} converges to a minimizer of the objective function $f+g$.
\item[$(iii)$] Sequences $\{x^k\}$ generated by Algorithm \eqref{algo3} with $\{y^k\}$ converges to a minimizer of the objective function $f+g\circ B+h$.
\item[$(iv)$] The local convergence rate of $\{x^k\}$ in the above all three cases is $o\left(k^{-\frac{1}{2}}\right)$.
\end{itemize}
\end{proposition}
\begin{proof}
By Proposition \ref{ToptSN2}, we know that $T_1$ and $T_2$ are both GAN with exponent 2 with respect to $\|\cdot\|_2$, $T_3$ is GAN with exponent 2 with respect to $\|\cdot\|_W$. Then we conclude from Theorem \ref{thm_SNconv} and the equivalence of all norms on $\bRn$ that the fixed-point iterations of $T_1$, $T_2$ and $T_3$ (or Algorithm \eqref{algo1}, \eqref{algo2} and \eqref{algo3}) converge to their fixed-points and $(iv)$ holds. Let $v^*:=\left(\begin{array}{cc}
x^*\\
y^*
\end{array}\right)$ be the fixed-point of $T_3$ that the fixed-point iteration of $T_3$ converges to, where $x^*\in\bRn$, $y^*\in\bRm$. The proof is completed by noticing that the fixed-points of $T_1$ and $T_2$ are minimizers of $f$ and $f+g$ respectively, and $x^*$ is a minimizer of $f+g\circ B+h$.
%
\end{proof}

We comment that we have recovered the local convergence rate $o\left(k^{-\frac{1}{2}}\right)$ of Algorithm \eqref{algo3} previously  obtained in \cite{li2016fast}, by showing that $T_3$ is the generalized averaged nonexpansiveness with exponent $2$.

Based on the convergence rate analysis in previous sections, we are able to obtain further convergence rate results for the fixed-point algorithms \eqref{algo1} and \eqref{algo2}. We first consider the one-dimensional case for Algorithm \eqref{algo1}.

\begin{proposition}\label{thm_graddesc1d}
Suppose that function $f\in\Gamma_0(\bR)$ is differentiable with an $L$-Lipschitz continuous derivative, where $L\in\bR_+$. Then for $\beta\in\left(0,\frac{2}{L}\right)$, the following hold:
\begin{itemize}
\item[$(i)$] $T_1$ is GAN with exponent 1.
\item[$(ii)$] For any initial vector $x^0\in\bRn$, the sequence $\{x^k\}$ generated by Algorithm \eqref{algo1} converges to a minimizer of $f$ with an $o\left(\frac{1}{k}\right)$ local convergence rate.
\end{itemize}
\end{proposition}
\begin{proof}
We first prove $(i)$. By the definition of generalized averaged nonexpansiveness with exponent 1 and the definition of $T_1$, it suffices to show that there exists $\mu\in\bR_{+}$ such that for all $x,y\in\bR$,
\begin{equation}\label{ineqwvSN1}
|(x-y)-\beta(f'(x)-f'(y))|+\mu|\beta(f'(x)-f'(y))|\leq|x-y|.
\end{equation}
Let $w:=x-y$, $v:=\beta(f'(x)-f'(y))$ and $\mu=\min\left\{\frac{1}{2},\frac{2}{\beta L}-1\right\}$. Then $\mu\in(0,1)$ and $L\leq\frac{2}{\beta(1+\mu)}$. It follows from the $L$-Lipschitz continuity of $f'$ that
\begin{equation}\label{ineq:vleqw}
|v|\leq\beta L|w|\leq\frac{2}{1+\mu}|w|.
\end{equation}
The convexity of $f$ implies that $f'$ is monotonically increasing, and hence $wv\geq0$. Multiplying $(1-\mu^2)|v|$ on both sides of \eqref{ineq:vleqw}, we obtain that
$$
(1-\mu^2)v^2\leq2(1-\mu)wv,
$$
which implies that
$$
v^2-2wv+w^2\leq w^2-2\mu wv+\mu^2v^2,
$$
that is,
\begin{equation}\label{ineqwvsquare}
(w-v)^2\leq(|w|-\mu|v|)^2.
\end{equation}
By \eqref{ineq:vleqw} and the fact that $\mu\in(0,1)$, it is easy to see that $\mu|v|\leq|w|$. Hence \eqref{ineqwvsquare} is equivalent to $|w-v|\leq|w|-\mu|v|$, that is, \eqref{ineqwvSN1} holds, and hence $T_1$ is GAN with exponent 1.

Now we employ $(i)$ and Theorem \ref{thm_SNconv} to prove $(ii)$. The convergence of $\{x^k\}$ to a minimizer of $f$ has been shown in Proposition \ref{thm_convT123} $(i)$. Since $T_1$ is GAN with exponent 1, the $o\left(\frac{1}{k}\right)$ local convergence rate of its fixed-point iteration follows from Theorem \ref{thm_SNconv} immediately.
\end{proof}

In Proposition \ref{thm_graddesc1d}, we have shown the generalized averaged nonexpansiveness with exponent 1 of $T_1$ and the convergence rate of Algorithm \eqref{algo1} in one-dimensional case. We next consider the higher-dimensional case.

In fact, we are able to show that $T_1$ is both GAN with exponent 1 and HR with exponent 1 under appropriate assumptions, which leads to an exponential global convergence rate for Algorithm \eqref{algo1} by Theorem \ref{thm_SNandBHR}. To establish this result, we recall the Baillon-Haddad theorem \cite{baillon1977quelques}.

\begin{lemma}\label{BHthm}
Suppose that $\psi:\bRn\to\bR$ is a differentiable convex function. Then $\nabla\psi$ is $L$-Lipschitz  with respect to $\|\cdot\|$ for some $L\in\bR_+$ if and only if
$$
\|\nabla\psi(x)-\nabla\psi(y)\|^2\leq L\langle x-y,\nabla\psi(x)-\nabla\psi(y)\rangle,\ \ \mbox{for all} \ \ x,y\in\bRn.
$$
\end{lemma}

\begin{theorem}\label{thm_IgradSN1}
Let $f\in\Gamma_0(\bRn)$ be differentiable. If there exist $L_1\geq L_2>0$ such that
\begin{equation}\label{doubleLipcont}
L_2\|x-y\|\leq\|\nabla f(x)-\nabla f(y)\|\leq L_1\|x-y\|,\ \ \mbox{for all}\ \ x,y\in\bRn,
\end{equation}
then for $\beta\in\left(0,\frac{2}{L_1}\right)$, the following hold:
\begin{itemize}
\item[$(i)$] $T_1$ is both GAN with exponent 1 and HR with exponent 1.
\item[$(ii)$] For any initial vector $x^0\in\bRn$, the sequence $\{x^k\}$ generated by Algorithm \eqref{algo1} converges to a minimizer of $f$ with an exponential global convergence rate.
\end{itemize}
\end{theorem}
\begin{proof}
We first prove the generalized averaged nonexpansiveness of $T_1$ by employing Theorem \ref{thm_TdbLipSN1}. Let $T:=\frac{1}{L_1}\nabla f$. It follows from the second inequality of \eqref{doubleLipcont} and Lemma \ref{BHthm} that
$$
\|Tx-Ty\|^2\leq\langle x-y,Tx-Ty\rangle,\ \ \mbox{for all}\ \ x,y\in\bRn,
$$
that is, $T$ is firmly nonexpansive. By the first inequality of \eqref{doubleLipcont}, we have
$$
\|Tx-Ty\|\geq\frac{L_2}{L_1}\|x-y\|,
$$
where $\frac{L_2}{L_1}\in(0,1]$. Since $T_1=\mI-\beta L_1 T$ and $\beta L_1\in\left(0,2\right)$, the generalized averaged nonexpansiveness with exponent 1 of $T_1$ follows from  Theorem \ref{thm_TdbLipSN1} immediately.

We next show the H$\ddot{\text{o}}$lder regularity of $T_1$. Let $\mu=\frac{1}{\beta L_2}$. Since $f:\bRn\to\bR$ is differentiable, by Fermat's lemma \cite{zorich2015Mathematical}, we know that $\nabla f(\hx)={\bm0}$ for any $\hx\in\Fix(T_1)$. Now using the first inequality of \eqref{doubleLipcont}, for any $x\in\bRn$, $\hx\in\Fix(T_1)$,
$$
\|x-\hx\|\leq\frac{1}{L_2}\|\nabla f(x)-\nabla f(\hx)\|=\mu\|\beta\nabla f(x)\|=\mu\|x-T_1x\|,
$$
which implies that $d(x,\Fix(T_1))\leq\mu\|x-T_1x\|$. Thus, $T_1$ is HR with exponent 1.

Now we employ $(i)$ and Theorem \ref{thm_SNandBHR} to prove $(ii)$. The convergence of $\{x^k\}$ to a minimizer of $f$ has been shown in Proposition \ref{thm_convT123} $(i)$. Since $T_1$ is both GAN with exponent 1 and HR with exponent 1, $(ii)$ follows from Theorem \ref{thm_SNandBHR} immediately.
\end{proof}

We next provide an example whose objective function satisfies \eqref{doubleLipcont}.
\begin{corollary}\label{LSexample}
Suppose function $f:\bRn\to\bR$ is defined by $f(x):=\frac{1}{2}\|Ax-b\|_2^2$, where $A\in\bR^{m\times n}$ is a full column rank matrix, $b\in\bRm$. Then for any initial vector $x^0\in\bRn$, the sequence $\{x^k\}$ generated by Algorithm \eqref{algo1} converges to the minimizer of $f$ with an exponential global convergence rate for $\beta\in\left(0,\frac{2}{L}\right)$, where $L$ is the maximum eigenvalue of $A^\top A$.
\end{corollary}
\begin{proof}
It is easy to verify that $f\in\Gamma_0(\bRn)$ and it is differentiable. The fact that $A$ has full column rank implies the positive definiteness of the Hessian matrix $H:=A^\top A$ of $f$. Hence $f$ is strictly convex and has a unique minimizer. According to Theorem \ref{thm_IgradSN1} , to prove this corollary, it suffices to show that there exist $L_1\geq L_2>0$ such that \eqref{doubleLipcont} holds. By the definition of $f$,
$$
\|\nabla f(x)-\nabla f(y)\|_2^2=z^\top H^\top Hz,
$$
where $z:=x-y$. Of course, $H^\top H\in\bRnn$ is positive definite. Let $0<\lambda_1\leq\lambda_2\leq\cdots\leq\lambda_n$ be the $n$ eigenvalues of $H^\top H$. Then $H^\top H-\lambda_1 I$ and $\lambda_n I-H^\top H$ are both positive semi-definite, which implies that
$$
\lambda_1\|z\|_2^2\leq\|Hz\|_2^2\leq\lambda_n\|z\|_2^2,
$$
that is, \eqref{doubleLipcont} holds by setting $L_2=\sqrt{\lambda_1}$ and $L_1=\sqrt{\lambda_n}$. Therefore, the desired result of this corollary follows from Theorem \ref{thm_IgradSN1} $(ii)$ immediately.
\end{proof}

To close this section, we present a local convergence rate for Algorithm \eqref{algo2}. Note that when the $\ell_2$ norm in the definition of generalized averaged nonexpansiveness is replaced by the $\ell_1$ norm (generalized averaged nonexpansiveness with respect to $\ell_1$ norm), Proposition \ref{prop_SNcompos} and Theorem \ref{thm_SNconv} still hold with the $\ell_2$ norms in them is replaced by the $\ell_1$ norms. Moreover, we have the following theorem. 

\begin{theorem}\label{thm_T2SN1}
Suppose that for $i\in\bN_n$, $f_i\in\Gamma_0(\bR)$ is differentiable with an $L_i$-Lipschitz continuous derivative, for some $L_i\in\bR_+$. If function $f:\bRn\to\bR$ is given by
$$
f(x):=f_1(x_1)+f_2(x_2)+\cdots+f_n(x_n),
$$
$g:=\lambda\|\cdot\|_1$, for $\lambda\in[0,\infty)$, and $\beta\in\left(0,\frac{2}{\max_{i\in\bN_n}\{L_i\}}\right)$, then the following statements hold:
\begin{itemize}
\item[$(i)$] Operator $T_2$ is GAN with exponent 1 with respect to $\|\cdot\|_1$.
\item[$(ii)$] For any initial vector $x^0\in\bRn$, the sequence $\{x^k\}$ generated by Algorithm \eqref{algo2} converges to a minimizer of $f+g$ with a local convergence rate $o\left(\frac{1}{k}\right)$ with respect to $\|\cdot\|_1$.
\end{itemize}
\end{theorem}
\begin{proof}
We first prove $(i)$. By Example \ref{SNproxabs} and Proposition \ref{thm_graddesc1d} $(i)$, we know that both $\prox_{\beta\lambda|\cdot|}$ and $\mI-\beta f_i'$ are GAN with exponent 1. This implies that both $\prox_{\beta g}$ and $\mI-\beta\nabla f$ are GAN with exponent 1 with respect to $\ell_1$ norm. Then, by the $\ell_1$ norm version of Proposition \ref{prop_SNcompos}, $T_2$ is GAN with exponent 1 with respect to $\ell_1$ norm.

Now we conclude from $(i)$ and the $\ell_1$ norm version of Theorem \ref{thm_SNconv} that the fixed-point iteration of $T_2$ converges to a minimizer of $f+g$ with the convergence rate $o\left(\frac{1}{k}\right)$ in terms of $\|x^{k+1}-x^k\|_1$, which completes the proof of $(ii)$.
\end{proof}

Theorem \ref{thm_T2SN1} establishes the local convergence rate $o\left(\frac{1}{k}\right)$ with respect to $\|\cdot\|_1$ for Algorithm \eqref{algo2} by employing the generalized averaged nonexpansiveness with exponent 1 with respect to $\ell_1$ norm. The same local convergence rate with respect to an inner product norm for Algorithm \eqref{algo2} has been shown in Theorem 3 of \cite{davis2016convergence}.

\section{Conclusions}
We have introduced the notion of the generalized averaged nonexpansive (GAN) operator, which allows us to study convergence and convergence rates of fixed-point iterations of GAN operators not covered by the existing theory of the averaged nonexpansive operators. The introduced notion provides a unified approach for analyzing the convergence and convergence rates of convex optimization algorithms. The convergence rate results of optimization algorithms obtained from this approach cover existing understanding and lead to new findings.


\bibliographystyle{siamplain}
\bibliography{references}
\end{document}